\documentclass[11pt]{amsart}
\usepackage[english]{babel}
\usepackage[
pdftex,hyperfootnotes]
{hyperref}
\usepackage[table]{xcolor}
\usepackage {graphicx}
\usepackage[utf8]{inputenc}
\usepackage[T1]{fontenc}
\usepackage[total={19cm,24.5cm},top=2cm, left=1.25cm]{geometry}
\usepackage{color}
\usepackage{mathtools,arydshln}

\hypersetup{pdfauthor={Gerardo Araznibarreta,   Carlos Álvarez-Fernández, Manuel Mañas, Juan Carlos García-Ardila, Francisco Marcellán}, pdftitle={Christoffel  transformations of matrix orthogonal polynomials}, pdfsubject={Mathematics, Mathematical Analysis},pdfkeywords={Matrix orthogonal polynomials, Block Jacobi matrices, Darboux--Christoffel transformation,  Block Cholesky decomposition, Block LU decomposition, quasi-determinants}
}
\mathtoolsset{centercolon}
\usepackage{tikz}

\usetikzlibrary{calc,shadows,shapes.callouts,shapes.geometric,shapes.misc}
\newcommand{\tikzmark}[1]{\tikz[overlay,remember picture] \node (#1) {};}
\newcommand{\DrawBox}[1][]{%
	\tikz[overlay,remember picture]{
		\draw[black,#1]
		($(left)+(-0.2em,0.9em)$) rectangle
		($(right)+(0.2em,-0.8em)$);}
}

\usepackage[toc,page]{appendix}
\usepackage{tocvsec2}

\usepackage[pagewise,switch]{lineno}
\usepackage{blindtext}

           \usepackage[scaled]{helvet} 
           \usepackage{courier} 
           \usepackage{eulervm} 
           \normalfont
           \usepackage[T1]{fontenc}

\newtheorem{coro}{{Corollary}}
\newtheorem{exa}{{ Example}}
\newtheorem{defi}{{ Definition}}
\newtheorem{teo}{Theorem}
\newtheorem{pro}{ Proposition }

\newtheorem{rem}{Remark}

\newcommand{\prodint}[1]{\left\langle{#1}\right\rangle}

\renewcommand{\d}{\operatorname{d}}

\newcommand{\diag}{\operatorname{diag}}

\newcommand{\Z}{\mathbb{Z}}
\newcommand{\R}{\mathbb{R}}
\newcommand{\C}{\mathbb{C}}
\newcommand{\I}{\mathbb{I}}

\begin{document}
 \title[Christoffel Transformations for MOPRL and Toda lattice]{Christoffel  transformations for\\ matrix orthogonal polynomials in the real line and\\ the non-Abelian 2D Toda lattice hierarchy}

 \author[C Álvarez-Fernández]{Carlos Álvarez-Fernández}
 \address{Departamento de Métodos Cuantitativos, Universidad Pontificia Comillas, Calle de Alberto Aguilera 23, 28015-Madrid, Spain }
 \email{calvarez@comillas.edu}

 \author[G Ariznabarreta]{Gerardo Ariznabarreta}
 \address{Departamento de Física Teórica II (Métodos Matemáticos de la Física), Universidad Complutense de Madrid, Ciudad Universitaria, Plaza de Ciencias 1,  28040-Madrid, Spain}
 \email{gariznab@ucm.es}
 \thanks{GA thanks financial support from the Universidad Complutense de Madrid  Program ``Ayudas para Becas y Contratos Complutenses Predoctorales en España 2011"}

\author[JC García-Ardila]{Juan Carlos García-Ardila}
\address{Departamento de Matemáticas, Universidad Carlos III de Madrid, Avenida Universidad 30, 28911 Leganés, Spain}
\email{jugarcia@math.uc3m.es }

\author[M Mañas]{Manuel Mañas}
 \address{Departamento de Física Teórica II (Métodos Matemáticos de la Física), Universidad Complutense de Madrid, Ciudad Universitaria, Plaza de Ciencias 1, 28040 Madrid, Spain}
 \email{manuel.manas@ucm.es}

\author[F Marcellán]{Francisco Marcellán}
\address{Departamento de Matemáticas, Universidad Carlos III de Madrid and Instituto de Ciencias Matemáticas (ICMAT), Avenida Universidad 30, 28911 Leganés, Spain}
\email{pacomarc@ing.uc3m.es }
 	\thanks{MM \& FM thanks financial support from the Spanish ``Ministerio de Economía y Competitividad" research project MTM2012-36732-C03-01,  \emph{Ortogonalidad y aproximación; teoría y aplicaciones}}


\begin{abstract}
  Given a  matrix polynomial $W(x)$,  matrix bi-orthogonal polynomials  with respect to the sesquilinear  form
\begin{align*}
\prodint{P(x),Q(x)}_{W}=\int P(x)W(x)\d\mu(x) (Q(x))^{\top}, \quad P,Q\in\R^{p\times p}[x],
\end{align*}
 where $\mu(x)$ is a matrix of Borel measures supported in some infinite subset of the real line, are considered. Connection formulas between the sequences of matrix bi-orthogonal polynomials with respect to $\prodint{\cdot, \cdot }_W$ and matrix polynomials orthogonal  with respect to $\mu(x)$ are presented. In particular, for the case of nonsingular leading coefficients of the perturbation matrix polynomial $W(x)$ we present a generalization of the Christoffel formula constructed in terms of the Jordan chains of $W(x)$. For perturbations  with a singular leading coefficient several examples by Durán et al are revisited.
 Finally, we extend these results to the non-Abelian 2D Toda lattice hierarchy.
\end{abstract}

\maketitle

\tableofcontents

\section{Introduction}
This paper is devoted to the extension of the Christoffel formula to the Matrix Orthogonal Polynomials on the Real Line (MOPRL) and the non-Abelian 2D Toda lattice hierarchy.
\subsection{Historical background and state of the art}
In 1858 the German mathematician Elwin Christoffel \phantomsection\cite{christoffel} was interested, in the framework of Gaussian quadrature rules,  in finding  explicit formulas relating the corresponding sequences of orthogonal polynomials with respect to two measures  $\d\mu$ (in the Christoffel's discussion was just the Lebesgue measure $\d\mu=\d x$ ) and  $d\hat{\mu}(x)= p(x) d\mu(x)$, with $p(x)=(x-q_1)\cdots(x-q_N)$ a signed polynomial in the support of $\d\mu$,  as well as the distribution of their zeros as nodes in such quadrature rules, see  \phantomsection\cite{Uva}. The so called Christoffel formula is a very elegant formula from a  mathematical point of view, and is a classical result which can be found in a number of orthogonal polynomials textbooks, see for example \cite{Chi,Sze,Gaut}. Despite these facts, we must mention that
 for computational and numerical purposes it is not so practical, see \phantomsection\cite{Gaut}.
These transformations have been extended from measures to the more general  setting of linear functionals.
In the theory of  orthogonal polynomials with respect to a moment linear functional  $u\in\big(\R[x]\big)'$, an element of the algebraic dual (which coincides with the topological dual)  of the linear space $\R[x]$ of polynomials with real coefficients. Given a  positive definite linear moment functional, i.e. $\Big|\big(\prodint{u,x^{n+m}}\big)_{n,m=0}^k\Big|>0$,
 $\forall k\in\Z_+:=\{0,1.2,\dots\}$, there exists a nontrivial probability  measure  $\mu$ such that  (see \phantomsection\cite{ahi}, \phantomsection\cite{Chi}, \phantomsection\cite{Sze}) $
\prodint{u,x^n}=\int x^m\d\mu(x)$.
 Given a moment linear functional $u$, its canonical or elementary Christoffel transformation is a new moment functional given by  $\hat{u}=(x-a)u$ with $ a\in\mathbb{R}$, see \phantomsection\cite{Bue1, Chi, Yoon}. The right inverse of a Christoffel transformation is called the Geronimus transformation. In other words, if you have a moment linear functional $u$, its elementary or canonical Geronimus transformation is a new moment linear functional $\check{u}$  such that $(x-a)\check{u}= u$. Notice that in this case $\check{u}$ depends on a free parameter, see \phantomsection\cite{Geronimus,Maro}. The right inverse of a general Christoffel transformation is said to be a multiple Geronimus transformation, see \phantomsection\cite{DereM}.
 All these transformations are refered as Darboux transformations, a name that was first given in the context of integrable systems in  \phantomsection\cite{matveev}. In 1878  the French mathematician Gaston Darboux, when studying the  Sturm--Liouville theory in  \phantomsection\cite{darboux2},  explicitly   treated these transformations, which appeared for the first time in \phantomsection\cite{moutard}.
 In the framework of orthogonal polynomials on the real line, such a factorization of Jacobi matrices has been studied in \phantomsection\cite{Bue1} and \phantomsection\cite{Yoon}. They also play an important role in the analysis of bispectral problems, see \phantomsection\cite{gru} and \phantomsection\cite{gru2}.

An important aspect of canonical Christoffel transformations is its relations with $LU$ factorization (and its flipped version, an $UL$ factorization) of the Jacobi matrix. A sequence of monic polynomials $\{P_{n}(x)\}_{n=0}^{\infty} $ associated with a nontrivial probability measure $\mu$ satisfies a three term recurrence relation (TTRR, in short)  $x P_{n}(x) = P_{n+1}(x) + b_{n} P_{n}(x) + a_{n}^{2}P_{n-1}(x), n\geq0,$ with the convention $P_{-1}(x)=0.$ If we denote by $P(x)= [P_{0}(x), P_{1}(x), \dots ]^{T}$, then the matrix representation of the multiplication operator by $x$ is directly deduced from the  TTRR and reads $x P(x)= J P(x)$, where $J$ is a tridiagonal semi-infinite matrix such that the entries in the upper diagonal are the unity. Assuming that $a$ is a real number off the support of $\mu,$ then you have a factorization  $J- aI= LU$, where $L$ and $U$ are, respectively, lower unitriangular and upper triangular matrices. The important observation is that the matrix $\hat{J}$ defined by  $\hat{J}- aI= UL$ is again a Jacobi matrix and the corresponding sequence of monic polynomials $\{\hat{P}_{n}(x)\}_{n=0}^{\infty}$ associated with the multiplication operator defined by $\hat{J}$ is orthogonal with respect to the canonical Christoffel transformation of the measure $\mu$ defined as above.

For a moment linear functional $u$, the Stieltjes  function
$S(x):=\sum_{n=0}^{\infty}\frac{\prodint{u,x^n}}{x^{n+1}}$ plays an important role in the theory of orthogonal  polynomials, due to its close relation with the measure associated to $u$ as well as its (rational) Padé Approximation, see \phantomsection\cite{Bre, Karl}. If you consider the canonical Christoffel transformation  $\hat{u}$ of the linear functional $u$, then its Stieltjes function is $\hat{S}(x)=(x-a)S(x)-u_0$. This is a particular case of the spectral linear transformations studied in \phantomsection\cite{Zhe}.

Given a bilinear form $L:\R[x]\times\R[x]\rightarrow \R$ one could consider
the following non-symmetric and symmetric  bilinear perturbations
\begin{align*}
\tilde{L}_1(p,q)=&L(w p,q),& \tilde{L}_2(p,q)=&L( p,w q), &\hat{L}(p,q)=&L(w p,w q),
\end{align*}
where $w(x)$ is a polynomial. The study of these perturbations can be found in  \phantomsection\cite{bue}.  Taking into account the matrix representation of the multiplication operator by $z$ is a Hessenberg matrix, the  authors establish a relation between the Hessenberg matrices associated  with the initial and the perturbed functional by using $LU$ and $QR$ factorization. 
They also give some algebraic relations between the sequences of orthogonal polynomials associated with those bilinear forms. The above perturbations can be seen as an extension of the Christoffel transformation for bilinear forms. When the bilinear form is defined by a nontrivial probability measure supported on the unit circle, Christoffel transformations have been studied in \phantomsection\cite{Cantero} in the framework of CMV matrices, i.e. the matrix representation of the multiplication operator by $z$ in terms of an orthonormal Laurent polynomial basis. Therein, the authors state the explicit relation between the sequences of orthonormal Laurent polynomials associated with a measure and its Christoffel transformation, as well as its link with QR factorizations of such CMV matrices.

The theory of scalar orthogonal polynomials with respect to probability measures supported either on the real line or the unit circle is a standard and classic topic in approximation theory and it also has remarkable applications in many domains as discrete mathematics, spectral theory of linear differential operators, numerical integration, integrable systems, among others. Some extensions of such a theory have been developed more recently. One of the most exciting generalizations appears when you consider non-negative Hermitian-valued matrix of measures of size $p\times p$ on a $\sigma$-algebra of subsets of a space $\Omega$ such that each entry is countably additive and you are interested in the analysis of the Hilbert space of matrix valued functions of size $p\times p$ under the inner product associated with such a matrix of measures. This question appears in the framework of weakly stationary processes, see \phantomsection\cite{Ros}. Notice that such an inner product pays the penalty of the non-commutativity of matrices as well as the existence of singular matrices with respect to the scalar case.
By using the standard Gram-Schmidt method for the canonical linear basis of the linear space of polynomials with matrix coefficients a theory of matrix orthogonal polynomials can be studied. The paper by M. G. Krein \phantomsection\cite{Krein} is credited as the first contribution in this topic. Despite they have been sporadically studied during the last half century, there is an exhaustive bibliography focused on inner products defined on the linear space of polynomials with matrix coefficients as well as on the existence and analytic properties of the corresponding  sequences of matrix orthogonal polynomials in the real line (see \phantomsection\cite{Dur5}, \phantomsection\cite{Dur4}, \phantomsection\cite{mir}, \phantomsection\cite{Rod}, \phantomsection\cite{van2}) and their applications in Gaussian quadrature for matrix-valued functions (\phantomsection\cite{van}), scattering theory (\phantomsection\cite{Ni}, \phantomsection\cite{Ger}) and system theory (\phantomsection\cite{Fuh}). The work  \phantomsection\cite{DAS} constitutes an updated overview on these topics.

But, more recently, an intensive attention was paid to the spectral analysis of second order linear differential operators with matrix polynomials as coefficients. This work was motivated by the Bochner's characterization of classical orthogonal polynomials (Hermite, Laguerre and Jacobi) as eigenfunctions of second order linear differential equations with polynomial coefficients. The matrix case gives a more rich set of solutions. From the pioneering work \phantomsection\cite{Dur3} some substantial progress has been done in the study of families of matrix orthogonal polynomials associated to second order linear differential operators as eigenfunctions and their structural properties (see \phantomsection\cite{Dur5}, \phantomsection\cite{grupach}, \phantomsection\cite{grupach2} as well as the survey \phantomsection\cite{Dur6}). Moreover, in \phantomsection\cite{Cas1} the authors showed that there exist sequences of orthogonal polynomials satisfying a first order linear matrix differential equation that constitutes a remarkable difference with the scalar case where such a situation does not appear. The spectral problem for second order linear difference operators with polynomial coefficients has been considered in \phantomsection\cite{Al} as a first step in the general approach. Therein four families of matrix orthogonal polynomials (as matrix relatives of Charlier, Meixner, Krawtchouk scalar polynomials and another one that seems not have any scalar relative) are obtained as illustrative examples of the method described therein. 

It is also a remarkable fact that matrix orthogonal polynomials appear in the analysis of non standard inner products in the scalar case. Indeed, from the study of higher order recurrence relations that some sequences of orthogonal polynomials satisfy (see \phantomsection\cite{Dur3} where the corresponding inner product is analyzed as an extension of the Favard's theorem and  \phantomsection\cite{Dur2}, where the connection with matrix orthogonal polynomials is stated), to the relation between standard scalar polynomials associated with measures supported on harmonic algebraic curves and matrix orthogonal polynomials deduced by a splitting process of the first ones (see \phantomsection\cite{MarSan}) you get an extra motivation for the study of matrix orthogonal polynomials. Matrix orthogonal polynomials appear in the framework of orthogonal polynomials in several variables when the lexicographical order is introduced. Notice that in such a case, the moment matrix has a  Hankel block matrix where each  block is a Hankel  matrix i.e. it has a doubly Hankel structure, see  \phantomsection\cite{delgado}.

Concerning spectral transformations, in \phantomsection\cite{DereM} the authors show that the so called multiple Geronimus transformations of a measure supported in the real line yield a simple Geronimus transformation for a matrix of measures. This approach is based on the analysis of general inner products $\langle \cdot,\cdot\rangle$ such that the multiplication by a polynomial operator $h$ is symmetric and satisfies an extra condition $\langle h(x)p(x), q(x)\rangle= \int p(x) q(x) \d\mu(x),$ where $\mu$ is a nontrivial probability measure supported on the real line. The connection between the Jacobi matrix associated to the sequence of scalar polynomials with respect to $\mu$ and the Hessenberg matrix associated with the multiplication operator by $h$ is given in terms of the so called $UL$ factorizations.  Notice that the connection between the Darboux process and the noncommutative bispectral problem has been discussed in \phantomsection\cite{gru1}. The analysis of perturbations on the entries of the matrix of moments from the point of view of the relations between the corresponding sequences of matrix orthogonal polynomials was done in \phantomsection\cite{Abdon_Luis}.

The seminal work of the Japanese mathematician Mikio Sato \phantomsection\cite{sato0,sato} and  later on of the Kyoto school \phantomsection\cite{date1,date2,date3} settled the basis for a  Grasmannian and Lie group theoretical description of integrable hierarchies.  Not much later Motohico Mulase \phantomsection\cite{mulase}  gave a mathematical description of factorization problems, dressing procedure, and linear systems as the keys for
integrability. It was not necessary to wait too long, in the development of integrable systems theory, to find multicomponent versions of the integrable Toda equations,  \phantomsection\cite{ueno-takasaki0,ueno-takasaki1,ueno} which later on played a prominent role in the connection with orthogonal polynomials and differential geometry. The multicomponent versions of the KP hierachy were analyzed in \phantomsection\cite{BtK,BtK2} and \phantomsection\cite{kac,manas-1,manas0}  and in  \phantomsection\cite{manas1,manas2}  we can find a further study of  the multi-component Toda lattice hierarchy, block Hankel/Toeplitz reductions, discrete flows, additional symmetries and dispersionless limits. For the relation with multiple orthogonal polynomials see \cite{adler 2,amu}.

The work of Mark Adler and Pierre van Moerbeke was fundamental  to the connection between integrable systems and orthogonal polynomials. They showed that the Gauss--Borel factorization problem is the keystone for this connection. In particular, their studies in the papers on the 2D Toda hierarchy and what they called the discrete KP hierarchy \phantomsection\cite{adler-van moerbeke, adler-vanmoerbeke 0,adler-van moerbeke 1,adler-van moerbeke 1.1,adler-van moerbeke 2} clearly established  --from a group-theoretical setup-- why standard orthogonality of polynomials and integrability of nonlinear equations of Toda type where so close.

The relation of multicomponent Toda  systems or non-Abelian versions of Toda equations with matrix orthogonal polynomials was studied, for example,  in \phantomsection\cite{mir,amu} (on the real line) and in \phantomsection\cite{mir2,ari0} (on the unit circle).

The approach to the Christoffel transformations in this paper, which is based on the Gauss--Borel factorization problem, has been used before in different contexts. It has been applied for the construction of discrete integrable systems connected with orthogonal polynomials of diverse types,
\begin{enumerate}
	\item  The case of multiple orthogonal polynomials and multicomponent Toda was analyzed in \phantomsection\cite{am}.
	\item In \phantomsection\cite{ari} we dealt with  the case of matrix orthogonal Laurent polynomials on the circle and CMV orderings.
	\item For orthogonal polynomials in several real variables see \phantomsection\cite{MVOPR,ari0} and  \phantomsection\cite{ari1} for orthogonal polynomials on the unit torus and the  multivariate extension of the CMV ordering.
\end{enumerate}
It is well known that there is a deep connection between discrete integrable systems and Darboux transformations of continuous integrable systems, see for example \phantomsection\cite{dsm}.
Finally, let us comment that, in the realm of several variables, in \phantomsection\cite{ari0,ari1,ari2} one can find extensions of the Christoffel formula to the multivariate scenario with real variables and on the unit torus, respectively.

\subsection{Objectives, results and layout of the paper}

In this contribution, we focus our attention on the study  of
Christoffel transformations (Darboux transformations in the language of integrable systems \phantomsection\cite{matveev}, or Lévy transformations in the language of differential geometry \phantomsection\cite{eisenhart}) for matrix sesquilinear forms. More precisely, given a matrix of measures $\mu(x)$   and   a matrix polynomial  $W(x)$ we are going to deal with the following matrix sesquilinear forms
\begin{align}\label{otro}
\prodint{P(x),Q(x)}_W=\int P(x)W(x)\d\mu(x)(Q(x))^\top.
\end{align}
We will first focus our attention on the existence of matrix bi-orthogonal polynomials with respect to the sesquilinear form  $\prodint{\cdot , \cdot}_W$ under some assumptions about the matrix polynomial $W$. Once this is done, the next step will be to find an explicit representation of such bi-orthogonal polynomials in terms of the matrix orthogonal polynomials with respect to the matrix of measures $\d\mu(x)$. We start with what we call connection formulas in Proposition \ref{pro:connection}.

One of the main achievements of this paper is Theorem \ref{theo:spectral} where we  extend the Christoffel formula to MOPRL with a perturbation given by an arbitrary degree monic matrix polynomial. For that aim we use the rich spectral theory available today for these type of polynomials, in particular tools like root polynomials and Jordan chains will be extremely useful, see \phantomsection\cite{Rod,Mark1}.
Following \phantomsection\cite{grupach,grupach2,Cas1} some applications to the analysis of matrix orthogonal polynomials which are eigenfunctions of second order linear differential operators and  related to polynomial perturbations of diagonal matrix of measures $\d\mu(x)$ leading to sesquilinear forms as $\prodint{\cdot , \cdot}_W$ will be considered.

Next, to have a better understanding of the singular leading coefficient case we concentrate on the study of some cases which generalize important examples given by Alberto Grünbaum and Antonio Durán in \phantomsection\cite{Dur6, Dur7}, in relation again with second order linear differential operators.

Finally, we see that these Christoffel transformations extend to more general scenarios in Integrable Systems Theory. In these cases we find the non-Abelian Toda hierarchy which is relevant in string theory. In general, we have lost the block Hankel condition, and we do not have anymore a matrix of measures but only a sesquilinear form. We show that Theorem \ref{theo:spectral} also holds in this general situation. At this point we must stress that for the non-Abelian Toda equation we can find Darboux transformations (or Christoffel transformations) in \phantomsection\cite{salle}, see also \phantomsection\cite{nimmo}, which contemplate only what it are called elementary transformations and their iteration. Evidently, their constructions do not cover by far what Theorem \ref{theo:spectral} does. There are many matrix polynomials that do not factor in terms of linear matrix polynomials and, therefore, they cannot be studied by means of the results in \phantomsection\cite{salle,nimmo}. We have been fortunate to have at our disposal the spectral theory of \phantomsection\cite{Rod,Mark1} that at the moment of the publication of \phantomsection\cite{salle} was not so well known and under construction.

The layout of the paper is as follows. We continue this introduction with two subsections that give the necessary background material  regarding the spectral theory of matrix polynomials and also of matrix orthogonal polynomials. Then, in \S 2, we
give the connection formulas for bi-orthogonal polynomials and  for the Christoffel–Darboux kernel, being this last result relevant to find the dual polynomials in the family of bi-orthogonal polynomials. We continue in \S 3 discussing the non singular leading coefficient case, i.e., the monic matrix polynomial perturbation.  We find  the Christoffel formula for matrix bi-orthogonal polynomials and, as an example, we consider the degree one monic matrix polynomial perturbations. We dedicate the rest of this section to discuss some examples. In \S 4 we start the exploration of the
singular leading coefficient matrix polynomial perturbations and,  despite we do not give a general theory, we have been able to
successfully discuss some relevant examples. Finally, \S 5 is devoted to the study of the extension of the previous results to the
non-Abelian 2D Toda lattice hierarchy.

\subsection{On spectral theory of matrix polynomials}\label{S:matrix polynomials}
Here we give some background material regarding matrix polynomials. For further reading we refer the reader to \phantomsection\cite{lan1}
\begin{defi}
Let  $A_0, A_1\cdots,A_N\in \mathbb{R}^{p\times p}$ be  square matrices with real entries.  Then
\begin{align}\label{eq:mp}
W(x)=A_N x^N+A_{N-1}x^{N-1}+\cdots +A_1x+A_0
\end{align}
is said to be  a  matrix polynomial of degree $N$, $\deg W=N$. The matrix polynomial  is said to be  monic when $A_N=I_p$,  where $I_p\in\R^{p\times p}$ denotes the identity matrix. The linear space of  matrix polynomials with coefficients in  $\mathbb{R}^{p\times p}$ will be denoted by
$\mathbb{R}^{p\times p}[x]$.
\end{defi}

\begin{defi}
	We say that a matrix polynomial $W$ as in \eqref{eq:mp} is monic normalizable if  $\det A_N\neq 0$ and say that
	$\tilde W(x):= A_N^{-1} W(x)$ is its monic normalization.
\end{defi}

\begin{defi}
The spectrum,  or the set of eigenvalues, $\sigma(W)$ of a matrix polynomial $W$ is the zero set of  $\det W(x)$, i.e.
	\begin{align*}
	\sigma(W):=Z(W)=\{a\in\C: \det W(a)=0\}.
	\end{align*}
\end{defi}
\begin{pro}
	A monic normalizable matrix polynomial $W(x)$, $\deg W=N$, has $Np$ (counting multiplicities) eigenvalues or zeros; i.e., we can write
\begin{align*}
\det W(x)=\prod_{i=1}^q(x-x_i)^{\alpha_i}
 \end{align*}
with $Np=\alpha_1+\dots+\alpha_q$.
\end{pro}

\begin{rem}
In   contrast with the scalar case, there exist  matrix polynomials which do not have a unique factorization in terms of degree one factors or even it could happen that the factorization does not exist.  For example, the matrix polynomial
\begin{align*}
 W(x)=I_2x^2-
\begin{bmatrix}
2&0\\
0&2
\end{bmatrix}
x
\end{align*}
can be written as
\begin{align*}
W(x)=\left(I_2x-
\begin{bmatrix}
1&-1\\
-1&1
\end{bmatrix}
\right)
 \left(I_2x-
\begin{bmatrix}
1&1\\
1&1
\end{bmatrix}
\right)
\ \ \ \ \ \ \text{or}\ \ \ \ \ \
W(x)=
\left(I_2 x-
\begin{bmatrix}
2&0\\
0&2
\end{bmatrix}
\right)I_2x,
\end{align*}
but the polynomial
\begin{equation*}
W(x)=I_2x^2-
\begin{bmatrix}
0&0\\
1&0
\end{bmatrix}
\end{equation*}
 cannot be factorized in terms of degree one matrix polynomials.
\end{rem}
\begin{defi}
\begin{enumerate}
	\item Two matrix polynomials $W_1,W_2\in \R^{m\times m}[x]$ are said to be equivalent $W_1\sim W_2$ if there exist
	two matrix polynomials $E,F \in \R^{m\times m}[x]$, with constant determinants (not depending on $x$), such that  $W_1(x)=E(x) W_2(x) F(x)$.
	\item 	A degree one matrix polynomial $I_{Np}x-A\in\R^{Np\times Np}$ is called a linearization of a monic matrix polynomial $W\in\R^{p\times p}[x]$ if
\begin{align*}
I_{Np}x-A\sim\begin{bmatrix}
W(x) & 0\\
0 & I_{(N-1)p}
\end{bmatrix}
\end{align*}
\end{enumerate}
\end{defi}

\begin{defi}
Given a matrix polynomial $W(x)= I_px^N+A_{N-1}x^{N-1}+\cdots +A_0$ its companion matrix  $C_1\in\mathbb{R}^{Np\times Np}$
is
\begin{align*}
C_1:=\begin{bmatrix}
0 & I_p & 0& \dots& 0\\
0 & 0 & I_p& \ddots  & 0\\
\vdots & \vdots &\ddots &\ddots& \\
0 &0 &0 & & I_p\\
-A_0 & -A_1 & - A_2 &\dots & -A_{N-1}
\end{bmatrix}.
\end{align*}
\end{defi}

The companion  matrix plays an important role in the study of the spectral properties of a matrix polynomial $W(x)$, see for example \phantomsection\cite{lan1,Mark2} and \phantomsection\cite{Mark1}

\begin{pro}
Given a monic matrix polynomial $W(x)=I_px^N +A_{N-1}x^{N-1}+\cdots +A_0$ its companion matrix  $C_1$
provides a linearization
\begin{align*}
I_{Np}x-C_1\sim \begin{bmatrix}
W(x) & 0\\
0 & I_{(N-1)p}
\end{bmatrix}
\end{align*}
where
\begin{align*}
E(x) &=\begin{bmatrix}
B_{N-1}(x) & B_{N-2}(x) &B_{N-3}(x) & \dots&B_1(x)&B_0(x)\\
-I_p & 0 &0 &\dots &0&0\\
0 &-I_p &0 &\dots &0&0\\
0 &0 &-I_p & & 0&0\\
\vdots& & &\ddots&\vdots\\
0&0&0& &-I_p &0
\end{bmatrix}, \\
F(x)&=\begin{bmatrix}
I_p & 0 &0& \dots & 0 &0\\
- I_p x & I_p & 0&\dots  &0 &0\\
0 & - I_px & I_p &\ddots &0 & 0\\
\vdots &  \ddots&\ddots  &\ddots &\vdots &\vdots\\
0 & 0 &0& \ddots& I_p &0\\
0 			&     0& 0 & & -I_px & I_p\\
\end{bmatrix},
\end{align*}
with $B_0(x):=I_p$, $B_{r+1}(x)=xB_r(x)+A_{N-r-1}$, for $r\in\{0,1,\dots,N-2\}$.
\end{pro}

From here one deduces the important
\begin{pro}
	The eigenvalues with multiplicities of a monic matrix polynomial coincide with those of its companion matrix.
\end{pro}

\begin{pro}\label{pro:partial multiplicity}
	Any  nonsingular matrix polynomial $W(x)\in \C^{m\times m}[x]$, $\det W(x)\neq 0$, can be represented
	\begin{align*}
W(x)=E(x_0)\operatorname{diag} ((x-x_0)^{\kappa_1},\dots ,(x-x_0)^{\kappa_m})F(x_0)
	\end{align*}at $x=x_0\in\C$,
	where $E(x_0)$ and $F(x_0)$ are nonsingular  matrices and $\kappa_1\leq\dots\leq\kappa_m$ are nonnegative integers. Moreover, $\{\kappa_1,\dots,\kappa_m\}$ are uniquely determined by $W$ and they are known as partial multiplicities of $W(x)$ at $x_0$.
\end{pro}

\begin{defi}
\begin{enumerate}
	\item  Given a monic matrix polynomial $W(x)\in\R^{p\times p}[x]$  with eigenvalues and multiplicities $\{x_k,\alpha_k\}_{k=1}^q$, $Np=\alpha_1+\dots+\alpha_q$, a non-zero vector $v_{k,0}\in \C^p$ is said to be an eigenvector with eigenvalue $x_k$ whenever
$W(x_k)v_{k,0}=0$, $v_{k,0}\in\operatorname{Ker} W(x_k)\neq \{0\}$.
	
\item 	A sequence of vectors $\{v_{i,0},v_{i,1}\ldots, v_{1,m_i-1}\}$ is said to be a Jordan chain of length $m_i$ corresponding to $x_i\in\sigma(W)$ if  $v_{0,i}$ is an eigenvector of $W(x_i)$ and
\begin{align*}
\sum_{r=0}^{j}\frac{1}{r!}	\frac{\operatorname{d}^rW}{\operatorname{d} x^r}	\Big|_{x=x_i}v_{i,j-r}&=0, & j&=\{0,\ldots,m_i-1\}.
\end{align*}
\item A root polynomial at an eigenvalue $x_0\in\sigma(W)$ of $W(x)$ is a non-zero vector polynomial  $v(x)\in\C^p[x]$ such that $W(x_0)v(x_0)=0$. The multiplicity of this zero will be denoted by $\kappa$.
\item The maximal length of a Jordan chain corresponding to  the eigenvalue $x_k$ is called the multiplicity of the eigenvector $v_{0,k}$ and is denoted by $m(v_{0,k})$.
\end{enumerate}
\end{defi}
	The above definition generalizes the  concept of Jordan chain for degree one matrix polynomials \phantomsection\cite{aba}.
\begin{pro}
	The Taylor expansion of a root polynomial at a given eigenvalue $x_0\in\sigma(W)$
	\begin{align*}
	v(x)=\sum_{j=0}^qv_j (x-x_0)^j
		\end{align*}
provides a Jordan chain $\{v_0,v_1,\dots,v_{\kappa-1}\}$.
\end{pro}
\begin{pro}
	\label{pro:canonical jordan}
	Given an eigenvalue $x_0\in\sigma(W)$, with multiplicity $s=\dim\operatorname{Ker}W(x_0)$, we can construct $s$ root polynomials
	\begin{align*}
	v_i(x)=&\sum_{j=0}^{\kappa_i-1}v_{i,j}(x-x_0)^j, & i&\in\{1,\dots,s\},
	\end{align*}
	where $v_i(x)$	is a root polynomial with the largest order $\kappa_i$  among all root polynomials whose eigenvector does not belong to
	$\C\{v_{1,0},\dots,v_{i-1,0}\}$.
\end{pro}

\begin{defi}\label{def:canonical Jordan}
	A canonical set of Jordan chains of the monic matrix polynomial $W(x)$ corresponding to the eigenvalue $x_0\in\sigma(W)$ is,
	in terms of the root polynomials described in Proposition \ref{pro:canonical jordan}, the  set of vectors
	\begin{align*}
	\{v_{1,0}\dots,v_{1,\kappa_1-1},\dots, v_{s,0}\dots,v_{s,\kappa_r-1}\}.
	\end{align*}
\end{defi}

	\begin{pro}
	For a monic matrix polynomial $W(x)$ the lengths $\{\kappa_1,\dots,\kappa_r\}$ of the Jordan chains in a canonical set of Jordan chains of $W(x)$ corresponding to the eigenvalue $x_0$, see Definition \ref{def:canonical Jordan}, are the nonzero  partial multiplicities of $W(x)$ at $x=x_0$ described in Proposition \ref{pro:partial multiplicity}.
	\end{pro}
	
\begin{defi}\label{def:adapted}
		For each eigenvalue $x_i\in\sigma(W)$, with multiplicity $\alpha_i$ and  $s_i=\dim \operatorname{Ker} W(x_i)$, we  choose a canonical set of Jordan chains
	\begin{align*}
	&\big\{v_{j,0}^{(i)},\dots,v_{j,\kappa_{j}^{(i)}-1}^{(i)}\big\}, & j &=1,\dots,s_i,
	\end{align*}
	and, consequently, with partial multiplicities satisfying  $\sum_{j=1}^{s_i}\kappa_j^{(i)}=\alpha_i$.
	Thus, we can  consider the following  adapted root polynomials
	\begin{align}\label{vecmil}
		v_{j}^{(i)}(x)=\sum_{r=0}^{\kappa_j^{(i)}-1}v_{j,r}^{(i)}(x-x_i)^r.
	\end{align}
\end{defi}
\begin{pro}\label{pro:adapted_root}
	Given a monic matrix polynomial $W(x)$ the adapted root polynomials given in Definition \ref{def:adapted} satisfy
	\begin{align*}
	\frac{\operatorname{d}^r}{\operatorname{d} x^r}	\Big|_{x=x_i}\big(W(x) v_j^{(i)}(x)\big)&=0, & r&=0,\dots,\kappa^{(i)}_j-1, &
	j&=1\dots,s_i.
	\end{align*}
\end{pro}

\subsection{On orthogonal matrix polynomials}\label{S:mop}

Recall that a  sesquilinear  form  $\prodint{\cdot,\cdot}$  on the linear  space $\mathbb{R}^{p\times p}[x]$ is a map
\begin{equation*}
\prodint{\cdot,\cdot}: \mathbb{R}^{p\times p}[x]\times\mathbb{R}^{p\times p}[x]\longrightarrow \mathbb{R}^{p\times p},
\end{equation*}
such that for any triple $P,Q,R\in  \mathbb{R}^{p\times p}[x]$ of matrix polynomials we have
\begin{enumerate}
\item  $\prodint{AP(x)+BQ(x),R(x)}=A\prodint{P(x),R(x)}+B\prodint{Q(x),R(x)}$, $\forall A,B\in\mathbb{R}^{p\times p}$.
\item $\prodint{P(x),AQ(x)+BR(x)}=\prodint{P(x),Q(x)}A^\top+\prodint{P(x),R(x)}B^\top$, $\forall A,B\in\mathbb{R}^{p\times p}$.
\end{enumerate}
Here $A^{\top}$ denotes the  transpose of $A$, an antiautomorphism of order two in the ring of matrices.

\begin{defi}
  A  sesquilinear form $\prodint{\cdot,\cdot}$ is said to be  non degenerate if the leading principal sub-matrices of the corresponding Hankel matrix of moments $M:=\left(\prodint{I_px^iI,I_p x^j}\right)^{\infty}_{i,j=0}$ are nonsingular, and nontrivial if $\prodint{\cdot,\cdot}$ is a symmetric matrix sesquilinear form and $\prodint{P(x),P(x)}$ is a  positive definite matrix  for all  $ P(x)\in \mathbb{R}^{p\times p}[x]$ with nonsingular leading coefficient.
\end{defi}	
	
Given a  sesquilinear  form $\prodint{\cdot,\cdot}$, two sequences of polynomials  $\big\{P_n^{[1]}(x)\big\}_{n=0}^\infty$ and $\big\{P_n^{[2]}(x)\big\}_{n=0}^\infty$ are said to be bi-orthogonal with respect to $\prodint{\cdot,\cdot}$ if
\begin{enumerate}
\item $\deg(P_n^{[1]})=\deg(P_n^{[2]})=n$ for all $n\in\Z_+$.
\item $\prodint{P_n^{[1]}(x),P_m^{[2]}(x)}=\delta_{n,m}H_n$ for all $n,m\in\Z_+$
\end{enumerate}
where $H_n\neq 0$ and $\delta_{n,m}$ is the Kronecker delta. Here, it is important to notice the order of the polynomials in the sesquilinear form; i.e., if $n\neq m$ then $\prodint{P_n^{[2]}(x),P_m^{[1]}(x)}$ could be different from $0$.

\begin{rem}
Recall that if A is a positive semidefinite (resp. definite) matrix, then  there exists a unique  positive  semidefinite (resp. definite) matrix $B$ such that $B^2=A$. $B$ is said to be  the square root of  $A$ (see \phantomsection\cite{hor}, Theorem 7.2.6) and we denote it by $B=:A^{1/2}$. As in the scalar case, when $\prodint{\cdot,\cdot}$ is  a sesquilinear form, we will write the matrix $\prodint{P,P}^{1/2}:=\|P\|$.
\end{rem}

Let
\begin{align*}
\mu=\begin{bmatrix}
\mu_{1,1}&\dots &\mu_{1,p}\\
\vdots & &\vdots\\
\mu_{p,1} &\dots&\mu_{p,p}
\end{bmatrix}
\end{align*}
be a  $p\times p$ matrix of Borel measures in $\R$.
 Given any pair of matrix polynomials $P(x),Q(x)\in\mathbb{R}^{p\times p}[x]$  we introduce the following     sesquilinear form
\begin{align*}
\prodint{P(x),Q(x)}=\int_\R P(x)\d\mu(x)(Q(x))^{\top}.
\end{align*}
In terms of the moments of the  matrix of measures $\mu$ we define the  matrix moments as
\begin{align*}
m_n:=\int_\R x^n\d\mu(x)\in\R^{p\times p}
\end{align*}
and arrange them in the semi-infinite block matrix and its  $k$-th truncation
\begin{align*}
M&:=\begin{bmatrix}
m_{0}&m_{1}& m_2&\cdots\\
m_{1}&m_{2}& m_3&\cdots\\
m_{2}&m_{3}& m_4&\cdots\\
\vdots    &\vdots      &\vdots&\cdots\\
\end{bmatrix}, &
M_{[k]}&:=\begin{bmatrix}
m_{0}&\cdots&m_{k-1}\\
\vdots& &\vdots\\
m_{k-1}&\cdots&m_{2k-2}
\end{bmatrix}.
\end{align*}
Following \phantomsection\cite{ari} we can prove
\begin{pro}\label{pro:fac}
	If $\det M_{[k]}\neq 0$ for $k\in\{1,2,\dots\}$, then there exists a unique Gaussian factorization of the moment matrix $M$ given by 	\begin{align*}
	M=S_1^{-1} H (S_2)^{-\top},
	\end{align*}
	where $S_1,S_2$ are lower unitriangular block matrices and $H$ is a diagonal block matrix
	\begin{align*}
	S_i&=\begin{bmatrix}
	I_p&0&0&\dots\\
	(S_i)_{1,0}& I_p&0&\cdots\\
	(S_i)_{2,0}& (S_i)_{2,1}&I_p&\ddots\\
	&&&\ddots
	\end{bmatrix}, &
	H&=\begin{bmatrix}
H_0&0&0&\cdots\\
0&H_1&0&\ddots\\
0&0& H_2&\ddots\\
\vdots&\vdots&\ddots&\ddots
	\end{bmatrix}, & i&=1,2,
	\end{align*}
	with $(S_i)_{n,m},H_n\in\R^{p\times p}$, $\forall n,m\in\{0,1,\dots\}$.
		If $\mu=\mu^\top$ then we are dealing with a  Cholesky block factorization with  $S_1=S_2$ and $H=H^\top$.
\end{pro}
For $l\geq k$ we will also use the following bordered truncated moment matrix
\begin{align*}
M_{[k,l]}^{[1]}&:=\left[\begin{array}{ccc}
m_{0}&  \cdots & m_{k-1} \\
\vdots                        &   & \vdots \\
m_{k-2}  &  \cdots & m_{ 2k-3}\\[1pt]
\hline
m_{l}& \dots & m_{l+ k-1}
\end{array}\right],
\end{align*}
where we have replaced the last row of blocks, $\begin{bmatrix}
m_{k-1}& \dots & m_{2k-2}\end{bmatrix}$, of the truncated moment matrix $M_k$ by the row of blocks $\begin{bmatrix}
m_l& \dots & m_{l+ k-1} \end{bmatrix}$. We also need a similar matrix but replacing the last block column of $M_k$ by a column of blocks as indicated
\begin{align*}
M_{[k,l]}^{[2]}&:=\left[
\begin{array}{ccc|c}
m_{0} &  \cdots & m_{k-2}&m_{l} \\
\vdots  &                      &  \vdots  & \vdots \\
m_{k-1}  &  \cdots & m_{2k-3}&m_{k+l-1}
\end{array}\right].
\end{align*}

Using  last quasi-determinants, see \phantomsection\cite{gelfand,olver}, we find
\begin{pro}\label{qd1}
	If the last quasi-determinants of the truncated moment matrices are nonsingular, i.e.,
	\begin{align*}
	\det  \Theta_*(M_{[k]})\neq& 0, & k=1,2,\dots,
	\end{align*}
	then the Gauss--Borel factorization  can be performed and the following expressions
	\begin{align*}
	H_{k}&=\Theta_*(M_{[k+1]}), &
	(S_1^{-1})_{k,l}&=\Theta_*(M^{[1]}_{[k,l+1]})\Theta_*(M_{ [l+1]})^{-1},&
	(S_2^{-1})_{k, l}&=\big(\Theta_*(M_{[l+1]})^{-1}\Theta_*(M^{[2]}_{[k,l+1]})\big)^\top, \end{align*}
	hold.
\end{pro}
\begin{defi}
We define  $\chi(x):=[I_p,I_p x,I_px^2,\dots]^\top$ and  the vectors of matrix polynomials $P^{[1]}=[P_0^{[1]},P_1^{[1]},\cdots]^\top$
and $P^{[2]}=[P_0^{[2]},P_1^{[2]},\dots]^\top$, where
\begin{align*}
P^{[1]}&:=S_1\chi(x), &  P^{[2]}&:=S_2\chi(x).
\end{align*}
\end{defi}
\begin{pro}
	The matrix polynomials $P_n^{[i]}(x)$ are monic and  $\deg P_n^{[i]}=n$, $i=1,2$.
\end{pro}
Observe that the moment matrix can be expressed as
\begin{align*}
M=\int_{\R} \chi(x)\d\mu(x)(\chi(x))^\top.
\end{align*}

\begin{pro}
	The families of monic matrix polynomials $\big\{P_n^{[1]}(x)\big\}_{n=0}^\infty$ and $\big\{P_n^{[2]}(x)\big\}_{n=0}^\infty$
 are bi-orthogonal
\begin{align*}
\prodint{P^{[1]}_n(x),P^{[2]}_m(x)}&=\delta_{n,m}H_n,& n,m&\in\Z_+.
\end{align*}
If $\mu=\mu^\top$ then  $P_n^{[1]}=P_n^{[2]}=:P_n$ which in turn conform an orthogonal set of monic matrix  polynomials
 \begin{align*}
 \prodint{P_n(x),P_m(x)}&=\delta_{n,m}H_n,& n,m&\in\Z_+,
 \end{align*}
 and we can write $\|P_n\|=H_n^{1/2}$.
 	These bi-orthogonal relations can be recasted as
 	\begin{align*}
 	\int_\R P_n^{[1]}(x)\d\mu(x)x^m&=\int_\R x^m \d\mu(x)(P_n^{[2]}(x))^\top= H_n\delta_{n,m}, & &m\leq n.
 	\end{align*}
\end{pro}
\begin{proof}
	From the definition of the polynomials and the factorization problem we get
\begin{align*}
\int _\R P^{[1]}(x)\d\mu(x)(P^{[2]}(x))^\top=\int_\R S_1\chi(x)\d\mu(x)\chi(x)^T{S_2}^\top=S_1M (S_2)^\top=H.
\end{align*}
\end{proof}

\begin{rem}
The matrix of measures $\d\mu(x)$ may undergo a similarity transformation, $\d\mu(x)\mapsto \d\mu_c(x)$, and be conjugate to
$\d\mu(x)=B^{-1}\d\mu_c(x) B$, where $B\in\R^{p\times p}$ is a nonsingular matrix. The relation between the orthogonal polynomials given by these two measures is easily seen to be
\begin{gather*}
\begin{aligned}
M&=B^{-1}M_cB, & S_1&=B^{-1} S_{c,1}B, & H&=B^{-1} H_{c}B, & (S_2)^\top&=B^{-1} (S_{c,2})^\top B,
\end{aligned} \\
\begin{aligned}
P^{[1]}_n &=B^{-1}P^{[1]}_{c,n}B, &(P^{[2]}_n) ^\top&=B^{-1}(P^{[2]}_{c,n})^\top B.
\end{aligned}
\end{gather*}
\end{rem}

The shift matrix is the following semi-infinite block matrix
\begin{align*}
\Lambda:=\begin{bmatrix}
0&I_p&0&0&\dots\\
0&0&I_p&0&\dots\\
0&0&0&I_p&\dots\\
\vdots&\ddots&\ddots&\ddots&\ddots\\
\end{bmatrix}
\end{align*}
which satisfies the important spectral property
\begin{align*}
\Lambda\chi(x)=x\chi(x).
\end{align*}
\begin{pro}
	The block Hankel symmetry of the moment matrix can be written as
	\begin{align*}
\Lambda M=M\Lambda^\top.
	\end{align*}
\end{pro}
Notice that this symmetry completely characterizes  Hankel block matrices.

\begin{pro}
	We have the following last quasi-determinantal expressions
	\begin{align*}
	P^{[1]}_n(x)&=\Theta_*
	\begin{bmatrix}
	m_{0}&m_1&\cdots&m_{n-1}&I_p\\
m_{1}& m_2 &\cdots&m_{n}&I_px\\
	\vdots&\vdots&&\vdots&\vdots\\
	m_{n-1}&m_{n}&\cdots&m_{2n-2}&I_px^{n-1}\\
m_{n}&m_{n+1}&\cdots&m_{2n-1}&I_px^{n}
	\end{bmatrix}, \\
	(P^{[2]}_n(x))^\top&=\Theta_*
	\begin{bmatrix}
	m_{0}&m_1&\cdots&m_{n-1}&	m_{n}\\
	m_{1}& m_2 &\cdots&m_{n}&m_{n+1}\\
	\vdots&\vdots&&\vdots&\vdots\\
	m_{n-1}&m_{n}&\cdots&m_{2n-2}&m_{2n-1}\\
I_p&I_px&\cdots&I_px^{n-1}&I_px^{n}
	\end{bmatrix}.
	\end{align*}
	\end{pro}

 Given two sequences of matrix  bi-orthonormal polynomials $\{P_k^{[1]}(x)\}_{k=0}^\infty$ and $\{P_k^{[2]}(x)\}_{k=0}^\infty$, with respect to $\prodint{\cdot,\cdot}$, we define the $n$-th Christoffel--Darboux kernel matrix polynomial
\begin{align}\label{eq:CD kernel}
K_{n}(x,y):=\sum_{k=0}^{n}(P_k^{[2]}(y))^\top H_k^{-1}P^{[1]}_k(x).
\end{align}
Named after \phantomsection\cite{christoffel,darboux} see also \phantomsection\cite{DAS}.

\begin{pro}[ABC Theorem]
	An Aitken--Berg--Collar type formula
	\begin{align*}
	K_n(x,y)=\big[I_p,\dots,I_px^n\big](M_n )^{-1}\begin{bmatrix}
	I_p\\\vdots\\I_py^n
	\end{bmatrix},
	\end{align*}
	holds.
\end{pro}
The scalar version was  rediscovered and popularized by Berg \phantomsection\cite{berg}, who found it in a paper of Collar \phantomsection\cite{collar}, who attributes it to his teacher, Aitken. As we are inverting a  Hankel block matrix we are dealing with a Hankel Bezoutian type expression. This is connected with the following

\begin{pro}[Christoffel--Darboux formula]\label{pro:CD formula}
	The Christoffel--Darboux kernel satisfies
	\begin{equation*}
	(x-y)K_n(x,y)=(P^{[2]}_{n}(y))^\top (H_n)^{-1}P^{[1]}_{n+1}(x)-(P^{[2]}_{n+1}(y))^\top (H_{n})^{-1}P^{[1]}_{n}(x).
	\end{equation*}
\end{pro}

\section{Connection formulas for Darboux transformations of Christoffel type}

Given a monic matrix polynomial   $W(x)$   of degree $N$ we consider a new matrix of measures of the form
\begin{align*}
\d\mu(x)\mapsto \d\hat\mu(x):=W(x)\d\mu(x)
\end{align*}
with the corresponding perturbed sesquilinear form
\begin{align*}
\prodint{P(x),Q(x)}_{W}=\int P(x)W(x)\d\mu(x)(Q(x))^\top.
\end{align*}
In the same way as above, the moment block matrix
\begin{align*}
\hat{M}:=\int \chi(x)W(x)\d\mu(x)(\chi(x))^\top
\end{align*}
is introduced.
Let us assume that the perturbed moment matrix admits a Gaussian factorization
	\begin{align*}
	\hat M=\hat S_1^{-1} \hat H (\hat S_2)^{-\top},
	\end{align*}
	where $\hat S_1,\hat S_2$ are lower unitriangular block matrices and $\hat H$ is a diagonal block matrix
	\begin{align*}
	\hat S_i&=\begin{bmatrix}
	I_p&0&0&\dots\\
	(\hat S_i)_{1,0}& I_p&0&\cdots\\
	(\hat S_i)_{2,0}& (\hat S_i)_{2,1}&I_p&\ddots\\
	&&&\ddots
	\end{bmatrix}, &
	\hat H&=\begin{bmatrix}
	\hat H_0&0&0&\cdots\\
	0&\hat H_1&0&\ddots\\
	0&0& \hat H_2&\ddots\\
	\vdots&\vdots&\ddots&\ddots
	\end{bmatrix}, & i&=1,2.
	\end{align*}
Then, we have the corresponding perturbed bi-orthogonal matrix polynomials
\begin{align*}
\hat P^{[i]}(x)&=\hat S_i\chi(x), & i&=1,2,
\end{align*}
with respect to the perturbed sesquilinear form $\prodint{\cdot,\cdot}_{W}$.

\begin{rem}\label{remark:non_monic}
	The discussion for monic matrix polynomial perturbations and perturbations with a matrix polynomial with non singular leading coefficients are equivalent.
Indeed, if  instead of a monic matrix polynomial we have a matrix polynomial $ \tilde W(x)=A_Nx^N+\dots+A_0$ with a nonsingular leading coefficient,  $\det A_N\neq 0$, then we could factor out  $A_N$, $\tilde W(x)= A_N  W(x)$, where $W$ is monic. The moment matrices are related by $\tilde M= A_N\hat M$ and, moreover, $\tilde S_1= A_N \hat S_1(A_N)^{-1}$, $\tilde H=A_N \hat H$, $\tilde S_2=\hat S_2$, and $\tilde P^{[1]}_k(x)=A_N P^{[1]}_k(x)(A_N)^{-1}$ as well as $\tilde P^{[2]}_k(x)=\hat P^{[2]}_k(x)$.
\end{rem}

\subsection{Connection formulas for bi-orthogonal polynomials}\label{S:connection}

\begin{pro}\label{pro:darboux}
	The moment matrix $M$ and the $W$ perturbed moment matrix $\hat M$ satisfy
	\begin{align*}
	\hat M=W(\Lambda)M.
	\end{align*}
\end{pro}
\begin{defi}
	Let us introduce the following semi-infinite matrices
	\begin{align*}
	\omega^{[1]}&:=\hat S_1W(\Lambda)  S_1^{-1}, & \omega^{[2]}&:=(S_2\hat S_2^{-1})^\top,
	\end{align*}
	which we call resolvent or connection matrices.
\end{defi}
\begin{pro}[Connection formulas]\label{pro:connection}
Perturbed and non perturbed bi-orthogonal polynomials are subject to the following  linear connection formulas
		\begin{align}\label{eq:connection}
		\omega^{[1]}P^{[1]}(x)=&\hat P^{[1]}(x)W(x),\\ \label{eq:con_kernel}
	P^{[2]}(x)=&(\omega^{[2]})^\top\hat P^{[2]}(x).
	\end{align}
\end{pro}

\begin{pro}\label{pro:relation_omega}
The following relations hold
\begin{align*}
\hat H \omega^{[2]}=\omega^{[1]} H.
\end{align*}
\end{pro}
\begin{proof}
From Proposition  \ref{pro:darboux} and the $LU$ factorization we get
\begin{align*}
\hat S_1^{-1}\hat H\hat S_2^{-\top}=W(\Lambda)  S_1^{-1} H S_2^{-\top},
\end{align*}
so that
\begin{align*}
\hat H( S_2\hat S_2^{-1})^\top=\hat S_1W(\Lambda)  S_1^{-1} H
\end{align*}
and the result follows.
\end{proof}

From this result we easily get that
\begin{pro}\label{pro:omega_form}
The resolvent matrix	 $\omega$ is a band  upper triangular block matrix with all the block superdiagonals above the $N$-th one equal to zero.
\begin{align*}
\omega^{[1]}=\begin{bmatrix}
\omega^{[1]}_{0,0} &\omega^{[1]}_{0,1} &\omega^{[1]}_{0,2}  &\dots &\omega^{[1]}_{0,N-1} &I_p&0&0&\dots \\
0&\omega^{[1]}_{1,1} &\omega^{[1]}_{1,2} &\dots &\omega^{[1]}_{1,N-1} &\omega^{[1]}_{1,N} &I_p&0&\dots\\
0&0&\omega^{[1]}_{2,2} &\dots &\omega^{[1]}_{2,N-1} &\omega^{[1]}_{2,N} &\omega^{[1]}_{2,N+1} &I_p&\ddots\\
 &\ddots&\ddots &\ddots& &&&\ddots&\ddots
\end{bmatrix}
\end{align*}
with
\begin{align}\label{eq:hatH_H_k}
\omega^{[1]}_{k,k}=\hat H_k (H_k)^{-1}.
\end{align}
\end{pro}
\subsection{Connection formulas for the Christoffel--Darboux kernel}
In order to relate the perturbed and non perturbed kernel matrix polynomials let us introduce the following truncation of the connection matrix $\omega$.
\begin{defi}
	We introduce the  lower unitriangular matrix $\omega_{(n,N)}\in\R^{Np\times Np}$ with
	\begin{align*}
	\omega_{(n,N)}:=
	\begin{cases}
	\begin{bmatrix}
	0&\dots&0&0&\dots &0\\
	\vdots & &&\vdots & \\
		0&\dots&0&0&\dots &0\\
	\omega^{[1]}_{0,n+1}& \dots& \omega_{0,N-1}&I_p &\ddots&0\\
	\vdots&&&&\ddots&\vdots\\
	\omega^{[1]}_{n,n+1}&\cdots&&&\omega^{[1]}_{n,n+N-1} &I_p
	\end{bmatrix}, & n< N,\\[4pt]
	\begin{bmatrix}
	I_p&0&\dots &0&0\\
	\omega^{[1]}_{n-N+2,n+1}& I_p &\ddots&0&0\\
	\vdots&\ddots&\ddots&\ddots&\vdots\\
	\vdots&&\ddots&I_p&0\\
	\omega^{[1]}_{n,n+1}&\cdots&&\omega^{[1]}_{n,n+N-1} &I_p
	\end{bmatrix}, & n\geq N,
	\end{cases}
	\end{align*}
	and the  diagonal block matrix
	\begin{align*}
	\hat H_{n,N}=\diag(\hat H_{n-N+1},\dots,\hat H_n).
	\end{align*}
\end{defi}
Then, we can write the important
\begin{teo}\label{theorem:perturbed CD}
	The perturbed and original Christoffel-Darboux kernels are related by the following connection formulas
	\begin{align*}
	K_{n}(x,y)+\Big[(\hat P_{n-N+1}^{[2]}(x))^\top,\dots,(\hat P_{n}^{[2]}(x))^\top\Big]
	(\hat H_{(n,N)})^{-1}\omega_{(n,N)}
	\begin{bmatrix}
	P^{[1]}_{n+1}(x)\\
	\vdots\\
	P^{[1]}_{n+N}(x)
	\end{bmatrix}=\hat K_{n}(x,y)W(x),
	\end{align*}
	by convention $\hat P_j^{[2]}=0$ whenever $j<0$.
\end{teo}
\begin{proof}
	Consider the truncation
	\begin{align*}
(	\omega^{[2]})_{[n+1]}:=\begin{bmatrix}
	I_p &\dots & 	\omega^{[2]}_{0,N-1} & 	\omega^{[2]}_{0,N}& 0&0&\dots&0\\
	0&I_p &\dots &	\omega^{[2]}_{1,N+1} & 	\omega^{[2]}_{1,N+2}& 0&\dots&0\\
	\vdots&&\ddots&     &&\ddots&\ddots\\
	&&     &&&\\
	0&0& & 0& I_p& \dots& &	\omega^{[2]}_{n-N,n} \\
	\vdots&\vdots&&     &\ddots&\ddots&&\vdots\\
	0&0&&     &&&I_p&	\omega^{[2]}_{n-1,n}\\
	0&0&&     &&&0&I_p
	\end{bmatrix}.
	\end{align*}
	Recalling  \eqref{eq:con_kernel} in the form $(P^{[2]}(y))^\top=(\hat P^{[2]}(y))^\top\omega^{[2]}$ we see that
	$\big((\hat P^{[2]}(y))_{[n+1]}\big)^\top(	\omega^{[2]})_{[n+1]}=\big((P^{[2]}(y))_{[n+1]}\big)^\top$ holds for the $n$-th truncations of  $ P^{[2]}(y))$ and $\hat P^{[2]}(y)$.
	Therefore,
	\begin{align*}
\big((\hat P^{[2]}(y))_{[n+1]}\big)^\top(	\omega^{[2]})_{[n+1]}(H_{[n+1]})^{-1}(P^{[1]}(x))_{[n+1]}=&\big((P^{[2]}(y))_{[n+1]}\big)^\top(H_{[n+1]})^{-1}(P^{[1]}(x))_{[n+1]}\\
	=&K_{n}(x,y).
	\end{align*}

	Now, we consider $(	\omega^{[2]})_{[n+1]}(H_{[n+1]})^{-1}(P^{[1]}(x))_{[n+1]}$ and  recall  Proposition \ref{pro:relation_omega} in the form
	\begin{align*}
	(	\omega^{[2]})_{[n+1]}\big(H_{[n+1]}\big)^{-1}=(\hat H_{[n+1]})^{-1}(	\omega^{[1]})_{[n+1]}
	\end{align*}
	 which leads to
		\begin{align*}
		 (	\omega^{[2]})_{n+1}(H_{n+1})^{-1}(P^{[1]}(x))_{n+1}=
	(\hat H_{[n+1]})^{-1}(	\omega^{[1]})_{[n+1]}(P^{[1]}(x))_{[n+1]}.
	\end{align*}
	Observe also that
	\begin{align*}
(	\omega^{[1]})_{[n+1]}(P^{[1]}(x))_{[n+1]}=(\omega^{[1]} P^{[1]}(x))_{[n+1]}-\begin{bmatrix}
	0_{(n-N)p\times p}\\ V_N(x)
	\end{bmatrix}
	\end{align*}
	with
	\begin{align*}
	V_N(x)=
	\omega_{(n,N)}
	\begin{bmatrix}
	P^{[1]}_{n+1}(x)\\
	\vdots\\
	P^{[1]}_{n+N}(x)
	\end{bmatrix}.
	\end{align*}
	Hence, recalling \eqref{eq:connection} we get
	\begin{align*}
(	\omega^{[1]})_{[n+1]}(P^{[1]}(x))_{[n+1]}=(\hat P^{[1]}(x))_{[n+1]}W(x)-\begin{bmatrix}
	0_{(n-N)p\times p}\\ V_N(x)
	\end{bmatrix},
	\end{align*}
	and consequently
	\begin{multline*}
	\big((\hat P^{[2]}(y))_{[n+1]}\big)^\top
	(\omega^{[2]})_{[n+1]}(H_{[n+1]})^{-1}(P^{[1]}(x))_{[n+1]}\\
		\begin{aligned}
		&=\big((\hat P^{[2]}(y))_{[n+1]}\big)^\top(\hat H_{[n+1]})^{-1}(\hat P^{[1]}(x))_{[n+1]}W(x)-\big((\hat P^{[2]}(y))_{[n+1]}\big)^\top(\hat H_{[n+1]})^{-1}\begin{bmatrix}
	0_{(n-N)p\times p}\\ V_N(x)
	\end{bmatrix}\\&
		=\hat K_{n+1}(x,y)W(y)-\big((\hat P^{[2]}(y))_{[n+1]}\big)^\top(\hat H_{[n+1]})^{-1}\begin{bmatrix}
		0_{(n-N)p\times p}\\ V_N(x)
		\end{bmatrix}.
	\end{aligned}
	\end{multline*}
\end{proof}
\section{Monic matrix polynomial perturbations}

In this section we study the case of  perturbations  by monic matrix polynomials $W(x)$, which is equivalent to matrix polynomials with nonsingular leading coefficients. Using the theory given in \S \ref{S:matrix polynomials} we are able to extend the celebrated Christoffel formula to this context.

\subsection{The Christoffel formula for matrix bi-orthogonal polynomials}

We are now ready to show how the perturbed set of matrix bi-ortogonal  polynomials $\{\hat P_n^{[1]}(x),\hat P_n^{[2]}(x),\}_{n=0}^\infty$ is related to the original set  $\{ P_n^{[1]}(x), P_n^{[2]}(x)\}_{n=0}^\infty$
\begin{pro}
	Let $v_{j}^{(i)}(x)$ be the adapted root polynomials of the monic matrix polynomial $W(x)$ given in \eqref{vecmil}. Then, for each eigenvalue $x_i\in\sigma(W)$, $i\in\{1,\dots,q\}$,
		\begin{align}\label{eq:con2}
		\omega^{[1]}_{k,k} \frac{\d^r(P_k^{[1]}v_{j}^{(i)})}{\d x^r}\bigg|_{x=x_i} +\dots+\omega^{[1]}_{k,k+N-1}\frac{\d^r(P_{k+N-1}^{[1]}v_{j}^{(i)})}{\d x^r}	\bigg|_{x=x_i} =-\frac{\d^r(P^{[1]}_{k+N}v_{j}^{(i)})}{\d x^r}\bigg|_{x=x_i} ,
		\end{align}
		for  $r=0,\dots,\kappa^{(i)}_j-1,$ and $j=1\dots,s_i$.
\end{pro}
\begin{proof}
From \eqref{eq:connection} we get
	\begin{align*}
	\omega^{[1]}_{k,k} P_k^{[1]}(x)+\dots+\omega^{[1]}_{k,k+N-1}P_{k+N-1}^{[1]}(x)+P_{k+N}(x)=\hat P^{[1]}_k(x)W(x).
	\end{align*}
Now, according to Proposition \ref{pro:adapted_root} we have
\begin{align}\label{eq:hatP-chains-0}
\begin{aligned}
\frac{\d^r}{\d x^r}	\bigg|_{x=x_i}
\big(\hat P^{[1]}_kWv_{j}^{(i)} \big)
&=\sum_{s=0}^r\binom{r}{s}\frac{\d^{r-s}\hat P^{[1]}_k}{\d x^{r-s}}\bigg|_{x=x_i}
\frac{\d^s(W v_{j}^{(i)})}{\d x^s}	\bigg|_{x=x_i} \\&=0,
\end{aligned}
\end{align}
for  $r=0,\dots,\kappa^{(i)}_j-1$ and $j=1\dots,s_i$.
\end{proof}

Recall that $\sum_{j=1}^{s_i}\kappa_j^{(i)}=\alpha_i$ and that the sum of all multiplicities $\alpha_i$ is $Np=\sum_{i=1}^q\alpha_i$, $q=\# \sigma(W)$.
\begin{defi}
\begin{enumerate}
	\item 	For each eigenvalue $x_i\in\sigma(W)$,  in terms of the adapted root polynomials $v_{j}^{(i)}(x)$  of the monic matrix polynomial $W(x)$ given in \eqref{vecmil}, we introduce  the vectors
\begin{align*}
\pi_{j,k}^{(r),(i)}:= \frac{\d^r(P_{k}^{[1]}v_{j}^{(i)})}{\d x^r}\bigg|_{x=x_i}\in\C^p
\end{align*}
and  arrange them in the  partial row matrices $\pi_k^{(i)}\in \C^{p\times \alpha_i p}$ given by
\begin{align*}
\pi_k^{(i)}=\Big[\pi^{(0),(i)}_{1,k},\dots, \pi^{(\kappa^{(i)}_1-1),(i)}_{1,k},\dots,\pi^{(0),(i)}_{s_i,k},\dots,\pi^{(\kappa_{s_i}^{(i)}-1),(i)}_{s_i,k}\Big].
\end{align*}
We collect all them as
\begin{align*}
\pi_k:=\big[\pi_k^{(1)},\dots, \pi_k^{(q)}\big]\in \C^{p\times N p}.
\end{align*}
Finally, we have
\begin{align*}
\Pi_{k,N}:=\begin{bmatrix} \pi_k\\
\vdots\\
\pi_{k+N-1}
\end{bmatrix}\in \C^{Np\times Np}.
\end{align*}
\item 	In a similar manner, we  define
	\begin{align*}
	\gamma_{j,n}^{(r),(i)}(y)&:= \frac{\d^r\big(K_{n}(x,y)v_{j}^{(i)}(x)\big)}{\d x^r}\bigg|_{x=x_i}\in\C^p[y],\\
	\gamma_n^{(i)}(y)&:=\big[\gamma^{(0),(i)}_{1,n}(y),\dots, \gamma^{(\kappa^{(i)}_1-1),(i)}_{1,n}(y),\dots,\gamma^{(0),(i)}_{s_i,n}(y),\dots,\gamma^{(\kappa_{s_i}^{(i)}-1),(i)}_{s_i,n}(y)\big]\in \C^{p\times \alpha_i }[y],
	\end{align*}
	and, as above,  collect all of them in
	\begin{align*}
	\gamma_n(y)=\big[\gamma_n^{(1)},\dots, \gamma_n^{(q)}\big]\in \C^{p\times N p}[y].
	\end{align*}
\end{enumerate}
\end{defi}

\begin{teo}[The Christoffel formula for matrix bi-orthogonal polynomials]\label{theo:spectral}
	The perturbed set of matrix bi-orthogonal polynomials $\{\hat P_k^{[1]}(x), \hat P_k^{[2]}(x)\}_{k=0}^\infty$, whenever $\det \Pi_{k,N}\neq 0$, can be written as the following last quasideterminant
	\begin{align}\label{eq:Christoffel1}
	\hat P_k^{[1]}(x)W(x)&=\Theta_*\left[\begin{array}{c|c}
	\Pi_{k,N} &\begin{matrix}
	P^{[1]}_k(x)\\ \vdots\\ P^{[1]}_{k+N-1}(x)	\end{matrix}\\  \hline
	\pi_{k+N}&P^{[1]}_{k+N}(x)
	\end{array}\right], \\\label{eq:Christoffel2}
		(\hat P_{k}^{[2]}(x))^\top(\hat H_k)^{-1}&=\Theta_*\left[\begin{array}{c|c}
	\Pi_{k+1,N} & \begin{matrix}
	0\\
	\vdots\\
	0\\
	I_p
	\end{matrix}\\\hline
	\gamma_k(x) &0
	\end{array}\right].	\end{align}
	Moreover, the new matrix squared norms are
	\begin{align}\label{eq:hat_H_k}
	\hat H_k&=\Theta_*\left[\begin{array}{c|c}
	\Pi_{k,N} &\begin{matrix}
	H_k\\ 0\\ \vdots\\ 0 \end{matrix}\\  \hline
	\pi_{k+N}&0
	\end{array}\right].
	\end{align}
\end{teo}	
\begin{proof}
	We assume that  $ P_j^{[2]}=0$ whenever $j<0$.
To prove \eqref{eq:Christoffel1} notice that from \eqref{eq:con2} one deduces for the rows of the connection matrix that
	\begin{align}\label{eq:the_clue}
	[\omega^{[1]}_{k,k},\dots,\omega^{[1]}_{k,k+n-1}]=-\pi_{k+N}(\Pi_{k,N})^{-1}.
	\end{align}
	Now, using \eqref{eq:connection} we get
	\begin{align*}
[\omega^{[1]}_{k,k},\dots,\omega^{[1]}_{k,k+N-1}]\begin{bmatrix}P^{[1]}_k(x)\\\vdots\\P^{[1]}_{k+N-1}(x)
\end{bmatrix}+P^{[1]}_{k+N}(x)=	\hat P^{[1]}_k(x)W(x).
	\end{align*}
and \eqref{eq:Christoffel1} follows immediately.

To deduce \eqref{eq:Christoffel2} for $k\geq N$ notice that Theorem \ref{theorem:perturbed CD} together with \eqref{eq:hatP-chains-0} yields
\begin{align*}
\gamma^{(r),(i)}_{j,k}(x)+
\Big[(\hat P_{k-N+1}^{[2]}(x))^\top,\dots,(\hat P_{k}^{[2]}(x))^\top\Big]
(\hat H_{(k,N)})^{-1}\omega_{(k,N)}
\begin{bmatrix}
\pi^{(r),(i)}_{j,k+1}\\
\vdots\\
\pi^{(r),(i)}_{j,k+N}
\end{bmatrix}=0
\end{align*}
for  $r=0,\dots,\kappa^{(i)}_j-1$ and $j=1\dots,s_i$. We arrange these equations in a matrix form to get
\begin{align*}
\gamma_k(x)+\Big[(\hat P_{k-N+1}^{[2]}(x))^\top,\dots,(\hat P_{k}^{[2]}(x))^\top\Big](\hat H_{(k,N)})^{-1}\omega_{(k,N)}\Pi_{k+1,N}=0.
\end{align*}
Therefore, assuming that $\det\Pi_{k+1,N}\neq 0$, we get
\begin{align*}
\Big[(\hat P_{k-N+1}^{[2]}(x))^\top,\dots,(\hat P_{k}^{[2]}(x))^\top\Big](\hat H_{(k,N)})^{-1}\omega_{(k,N)}=-\gamma_k(x)(\Pi_{k+1,N})^{-1},
\end{align*}
which, in particular, gives
\begin{align*}
(\hat P_{k}^{[2]}(x))^\top(\hat H_k)^{-1}=-\gamma_k(x)(\Pi_{k+1,N})^{-1}
\begin{bmatrix}
0\\
\vdots\\
0\\
I_p
\end{bmatrix}.
\end{align*}
Finally, \eqref{eq:hat_H_k} is a consequence of \eqref{eq:hatH_H_k} and \eqref{eq:the_clue}.
\end{proof}

\subsection{Degree one monic matrix polynomial perturbations}
Let us illustrate the situation with the most simple case of a  perturbation of degree one monic polynomial matrix
\begin{align*}
W(x)=I_px-A.
\end{align*}
The spectrum $\sigma (I_px-A)=\sigma(A)=\{x_1,\dots,x_q\}$  is determined by the zeroes of the characteristic polynomial of $A$
	\begin{align*}
	\det(I_px-A)=(x-x_1)^{\alpha_1}\cdots(x-x_q)^{\alpha_q},
	\end{align*}
	and for each eigenvalue let  $s_i=\dim\operatorname{Ker} (I_px_i-A)$ be the corresponding geometric multiplicity, and $\kappa^{(i)}_j$, $j=1,\dots,s_i$, its partial multiplicities, so that $\alpha_i=\sum_{j=1}^{s_i}\kappa_j^{(i)}$ is the algebraic multiplicity (the order of the eigenvalue as a zero of the characteristic polynomial of $A$).
	After a similarity transformation of $A$ we will get its canonical Jordan form. With no lack of generality we assume that $A$ is already given in  Jordan canonical form
	\begin{align*}
	\left[\begin{array}{ccccccc}
	\tikzmark{left}\mathcal J_{\kappa^{(1)}_1}(x_1)& & 0& & & &\\
	&\ddots & & & &0 &\\
	0& & \mathcal J_{\kappa^{(1)}_{s_1}}(x_1)\tikzmark{right}\DrawBox[thick]& & & &\\
	& & &\ddots & & &\\
	& & & & \tikzmark{left}	\mathcal J_{\kappa^{(q)}_1}(x_q)& &0\\
	& 0 & & & & \ddots&\\
	& & & &0 &  &\mathcal J_{\kappa^{(q)}_{s_q}}(x_q)\tikzmark{right}\DrawBox[thick]
	\end{array}\right],
	\end{align*}
	where the Jordan blocks corresponding to each eigenvalue are given by
	\begin{align*}
	\mathcal J_{\kappa_j^{(i)}}(x_i)&:=
	\begin{bmatrix}
	x_i&        1&  0         &           &    \\
	0&x_i&    1      &           &     \\
	&         &\ddots     &\ddots     &    \\
	&         &           &x_i &  1\\
	&         &           &          0 &x_i
	\end{bmatrix}\in\R^{\kappa_j^{(i)}\times \kappa_j^{(i)}}, & j&=1,\dots, s_i.
	\end{align*}
	For each eigenvalue $x_i$ we pick  a basis  $\big\{v^{(i)}_{0,j}\big\}_{j=1}^{s_i}$ of $\operatorname{Ker}(I_px_i-A)$, then look
	for vectors $\{v_{r,j}^{(i)}\}_{r=1}^{\kappa_j^{(i)}-1}$	such that
	\begin{align*}
	(I_px_i-A)v_{r,j}^{(i)}=&-v_{r-1,j}^{(i)}, & r&=1,\dots,\kappa_j^{(i)}-1,
	\end{align*}
	so that $\{v^{(i)}_{r,j}\}_{r=0}^{\kappa^{(i)}_j}$ is a Jordan chain.
	As we are dealing with $A$ in its canonical form the vectors $v^{(i)}_{r,j}$ can be identified with those of the canonical basis
	$\{e_i\}_{i=1}^p$ of $\R^p$ with $e_i=(0,\cdots 1_i,\cdots 0)_{p}^\top $. Indeed, we have
	\begin{align*}
	v^{(i)}_{r,j}=e_{\alpha_1+\dots+\alpha_{i-1}+ \kappa^{(i)}_1+\dots +\kappa^{(i)}_{j-1}+r+1}.
	\end{align*}
	Then, consider the polynomial vectors
	\begin{align*}
	v^{(i)}_j(x)&=\sum_{r=0}^{\kappa^{(i)}_j-1}v^{(i)}_{r,j}(x-x_i)^r, & j&=1,\dots,s_i,& i&=1\dots,q,
	\end{align*}
	which satisfy
	\begin{align*}
	\frac{\d^r}{\d x^r}\bigg|_{x=x_i}\Big((I_px-A)v^{(i)}_{j}\Big)&=0, & r&=0,\dots,\kappa^{(i)}_j-1, & j&=1,\dots,s_i,& i&=1,\dots,q.
	\end{align*}	
	Consequently, 	
	\begin{align}\label{eq:jordan_normal}
	\frac{\d^r}{\d x^r}\bigg|_{x=x_i}\Big(\hat P^{[1]}(x)(I_px-A)v^{(i)}_{j}\Big)&=0, & r&=0,\dots,\kappa^{(i)}_j-1, & j&=1,\dots,s_i,& i&=1,\dots,q.
	\end{align}

Now, let us notice the following simple fact
\begin{pro}
	For a given matrix $A\in\R^{p\times p}$ any matrix polynomial $P(x)=\sum_{k=0}^nP_k x^k\in\R^{p\times p}[x]$, $\deg P=n$,  can be written as
	\begin{align*}
	P&=\sum_{k=0}^nP_{k}^{(A)} (I_px-A)^k.
	\end{align*}
	In particular, we have $P_{0}^{(A)}=P(A):=\sum_{k=0}^nP_k A^k$.
\end{pro}

      \begin{pro}\label{pro:P(A)}
      	We can write
      	\begin{align*}
      \frac{1}{r!}	\frac{\d^r	(P v^{(i)}_{j} )}{\d x^r}\bigg|_{x=x_i}
           	&=P(A)v^{(i)}_{r,j}, &	r&=0,\dots,\kappa^{(i)}_j-1,  &j&=1,\dots,s_i, & i&=1,\dots,q.
      	\end{align*}
      \end{pro}
\begin{proof}
	Observe that
		\begin{align*}
		\frac{\d^s	P}{\d x^s}\bigg|_{x=x_i}v^{(i)}_{r-s,j}&= \sum_{l=r}^{k}\frac{l!}{(l-s)!}P^{(A)}_l(I_px_i-A)^{l-s} v^{(i)}_{r-s,j}\\
		&=\sum_{l=s}^r(-1)^{l-s}\frac{l!}{(l-s)!}P^{(A)}_lv^{(i)}_{r-l,j}, & s&=1,\dots,r, & j&=1,\dots,s_i, &i=1,\dots,q,
		\end{align*}
		and, consequently,
		\begin{align*}
		\frac{\d^r	(P v^{(i)}_{j} )}{\d x^r}\bigg|_{x=x_i}
		&=\sum_{s=0}^{r}\frac{r!}{s!}\sum_{l=s}^r(-1)^{l-s}\frac{l!}{(l-s)!}P^{(A)}_lv^{(i)}_{r-l,j}\\
		&=r!\sum_{s=0}^{r}\sum_{l=s}^r(-1)^{l-s} \binom{l}{s}P^{(A)}_lv^{(i)}_{r-l,j}\\
		&=r!\sum_{m=0}^{r}(-1)^{r-m} \Big[\sum_{s=0}^{r-m}(-1)^s\binom{r-m}{s}\Big]P^{(A)}_{r-m}v^{(i)}_{m,j}\\
		&=r!P(A)v^{(i)}_{r,j},
		\end{align*}
		for $r=0,\dots,\kappa^{(i)}_j-1$,  $j=1,\dots,s_i$ and $i=1,\dots,q$.
\end{proof}
In the following we'll use
\begin{align*}
K_{n}(y,A):=\sum_{m=0}^n P^{[2]}_m(y)H_m^{-1} P^{[1]}_m(A).
\end{align*}

\begin{pro}[Degree one Christoffel formula]\label{pro:darboux_degree_one}
	If $W(x)=I_px-A$ and $\det P_n^{[1]}(A)\neq 0$ for $n\in \Z_+$, then the Christoffel formulas can be written as
	\begin{align*}
	\hat P^{[1]}_n(x)(I_px-A)&=\Theta_*\begin{bmatrix}
	P_n^{[1]}(A)& P_n^{[1]}(x)\\
		P_{n+1}^{[1]}(A)& P_{n+1}^{[1]}(x)
	\end{bmatrix}&
		(\hat P_{n}^{[2]}(y))^\top&=\Theta_*\begin{bmatrix}
		P_{n+1}^{[1]}(A)&I\\
		K_{n}(y,A)&0
		\end{bmatrix}\\
		&= P_{n+1}^{[1]}(x)-P_{n+1}^{[1]}(A)\big[P_n^{[1]}(A)\big]^{-1}P_n^{[1]}(x), & &
		= -K_{n}(y,A)\big[P_n^{[1]}(A)\big]^{-1}.
	\end{align*}
	For the perturbed matrix squared norms we have
	\begin{align*}
	\hat H_n&=\Theta_*\begin{bmatrix}
	P^{[1]}(A) &H_n\\P^{[n+1]}(A)&0
	\end{bmatrix}\\&=
	-P_{n+1}^{[1]}(A)\big[P_n^{[1]}(A)\big]^{-1}H_n.
	\end{align*}
\end{pro}
\begin{proof}
According to \eqref{eq:connection} and Theorem \ref{theorem:perturbed CD}
	 \begin{align*}
\omega_{n,n} P_n^{[1]}(x)+P_{n+1}^{[1]}(x)&=\hat P_n^{[1]}(x)(I_px-A), &
K_{n}(x,y)+(\hat P_{n}^{[2]}(y))^\top P^{[1]}_{n+1}(x)&=\hat K_{n}(x,y)(I_px-A),
\end{align*}	
and using \eqref{eq:jordan_normal} we conclude
\begin{align*}
\omega_{n,n} \frac{\d^r (P_k^{[1]}v^{(i)}_{j} )}{\d x^r}\bigg|_{x=x_i}+ \frac{\d^r (P_{n+1}^{[1]}v^{(i)}_{j} )}{\d x^r}\bigg|_{x=x_i}&=0,\\
\frac{\d^r }{\d x^r}\bigg|_{x=x_i}\big((K_n(x,y)v^{(i)}_{j} (x)\big)+(\hat P_{n}^{[2]}(y))^\top\frac{\d^r (P_{n+1}^{[1]}v^{(i)}_{j} )}{\d x^r}\bigg|_{x=x_i}&=0,
\end{align*}
		for $r=0,\dots,\kappa^{(i)}_j-1$,  $j=1,\dots,s_i$ and $i=1,\dots,q$.
From Proposition \ref{pro:P(A)} and the fact that the ordered arrangement of the Jordan chain vectors $ v^{(i)}_{r,j}$ gives the identity matrix $I_p$ we conclude
\begin{align*}
\omega_{n,n}P^{[1]}_k(A)+P^{[1]}_{n+1}(A)&=0,\\
K_n(y,A)+(\hat P_{n}^{[2]}(y))^\top P^{[1]}_{n+1}(A)&=0.
\end{align*}
\end{proof}

	 We now illustrate the Christoffel formula in the matrix orthogonal polynomial context with a  simple case.
		We will study what is the effect of the Christoffel transformation on a  positive  Borel scalar measure  $\d\mu(x)$, thus the perturbed matrix of measures is $(I_2x-A)\d\mu(x)$.
		The perturbed monic orthogonal polynomials will be expressed, see Proposition \ref{pro:darboux_degree_one},
		\begin{align}\label{eq:darboux_example}
		\hat P^{[1]}_n(x)&=\Big(I_2p_{n+1}(x)-p_{n+1}(A)\big(p_n(A)\big)^{-1}p_n(x)\Big)(I_2x-A)^{-1},\\
		\hat P^{[2]}_n(x)&=	-K_{n}(x,A)\big[P_n^{[1]}(A)\big]^{-1}
		\end{align}
		where $p_n(x),K_n(x,y)$ are the  scalar orthogonal polynomials and kernel polynomials associated with the original scalar positive Borel measure $\d\mu(x)$. Observe that despite  starting with a set of orthogonal polynomials the perturbation generates a set of bi-orthogonal matrix polynomials. As the original measure is scalar, if we ensure that $A=A^\top$ is symmetric, we will get $\hat P_n(x):=P^{[1]}_n(x)=P^{[2]}_n(x)$, a new set  of orthogonal matrix polynomials.
		
		However, this will be a very trivial situation as we have
		\begin{pro}\label{pro:diagonal_reducibility}
			The matrix orthogonal polynomials $\{\hat P_n(x)\}_{n=0}^\infty$ of the matrix of measures $(I_px-A)\d\mu(x)$, where $A=A^\top$ is symmetric and $\d\mu$ is a positive Borel scalar measure are similar to diagonal matrix orthogonal polynomials.
		\end{pro}
	\begin{proof}
		Being the matrix $A$ symmetric it will be always diagonalizable
		\begin{align*}
		A=Q D Q^\top,
		\end{align*}
		where $Q$ is an orthogonal matrix $Q^\top=Q^{-1}$ and $D=\diag(x_1,\dots,x_p)$, is a diagonal matrix that collects the eigenvalues, not necessarily different, of $A$.
		
		At the end, the new orthogonal polynomials will be
	\begin{align}\label{eq:darboux_diagonal}
	\hat P_n(x)
	&=Q\begin{bmatrix}\frac{p_{n+1}(x)-\frac{p_{n+1}(x_1)}{p_n(x_1)}p_n(x)}{x-x_1}&0&\dots&0\\
0&	\frac{p_{n+1}(x)-\frac{p_{n+1}(x_2)}{p_n(x_2)}p_n(x)}{x-x_2}& & 0\\
	\vdots &\ddots&\ddots &\vdots\\	
	\\0&0&&\frac{p_{n+1}(x)-\frac{p_{n+1}(x_p)}{p_n(x_p)}p_n(x)}{x-x_p}
	\end{bmatrix}Q^\top,
	\end{align}
	and the result is proven.
	\end{proof}
	Thus, we have a diagonal bunch of elementary Darboux transformations of the original scalar orthogonal polynomials associated to the scalar measure $\d\mu$. This situation reappears even when the matrix is not symmetric but diagonalizable, as we will have again that the perturbed matrix orthogonal polynomials will be similar to a similar bunch of elementary Darboux transformations of the original scalar orthogonal polynomials.
		\begin{align*}
		\hat P^{[1]}_n(x)
		&=Q\begin{bmatrix}\frac{p_{n+1}(x)-\frac{p_{n+1}(x_1)}{p_n(x_1)}p_n(x)}{x-x_1}&0&\dots&0\\
		0&	\frac{p_{n+1}(x)-\frac{p_{n+1}(x_2)}{p_n(x_2)}p_n(x)}{x-x_2}& & 0\\
		\vdots &&\ddots &\vdots\\	
		\\0&0&&\frac{p_{n+1}(x)-\frac{p_{n+1}(x_p)}{p_n(x_p)}p_n(x)}{x-x_p}
		\end{bmatrix}Q^{-1},\\
			\hat P^{[2]}_n(x)
			&=-Q\begin{bmatrix}\frac{K_n(x,x_1)}{p_n(x_1)}&0&\dots&0\\
			0&\frac{K_n(x,x_2)}{p_n(x_2)}& & 0\\
			\vdots &&\ddots &\vdots\\	
			\\0&0&&\frac{K_n(x,x_p)}{p_n(x_p)}
			\end{bmatrix}Q^{-1},
		\end{align*}
	where $Q$ does not need to be an orthogonal matrix.
	
	If the matrix is not diagonalizable and has  nontrivial Jordan blocks the situation is different. Let us explore this case
	when $p=2$. Hence, we consider
	\begin{align*}
	W(x)&=I_2x-A,
	\end{align*}
	with
	\begin{align*}
	A&=M\begin{bmatrix}
	x_1 & 1\\ 0& x_1
	\end{bmatrix}M^{-1}.
	\end{align*}

Now we have only one eigenvalue $\sigma(A)=\{x_1\}$, with a length 2 Jordan chain.
Thus, there is a linear basis  $\{v_1,v_2\}\subset \R^2$, $(A-x_1)v_1=v_2$,  $(A-x_1)v_2=0$, with $v_i=[v_{i,1}\quad v_{i,2}]^\top$, $i\in\{1,2\}$, such that
\begin{align*}
M=\begin{bmatrix}
v_{1,1} &v_{2,1}\\
v_{1,2} & v_{2,2}
\end{bmatrix}.
\end{align*}

Therefore,
\begin{align*}
p_{n+1}(A)\big(p_n(A)\big)^{-1}&= M\begin{bmatrix}
\frac{p_{n+1}(x_1)}{p_n(x_1)} &	\frac{W(p_n,p_{n+1})(x_1)}{(p_n(x_1))^2}\\ 0&\frac{p_{n+1}(x_1)}{p_n(x_1)}
\end{bmatrix} M^{-1},&
(I_2x-A)^{-1}&=M\begin{bmatrix}
\frac{1}{x-x_1} & \frac{1}{(x-x_1)^2}\\0& 	\frac{1}{x-x_1}
\end{bmatrix}M^{-1},
\end{align*}
where
\begin{align*}
W(p_n,p_{n+1})(x)=p_n(x) p'_{n+1}-p_{n+1}(x)p'_{n}(x)
\end{align*}
is the Wronskian of two consecutive orthogonal polynomials.

Hence,
	\begin{align}\label{eq:Jordan_example}
	\hat P^{[1]}_n(x)
	&=M\begin{bmatrix}\frac{p_{n+1}(x)-\frac{p_{n+1}(x_1)}{p_n(x_1)}p_n(x)}{x-x_1}&
	\frac{p_{n+1}(x)-\frac{p_{n+1}(x_1)}{p_n(x_1)}p_n(x)-\frac{W(p_n,p_{n+1})(x_1)}{(p_n(x_1))^2}(x-x_1)p_n(x)}{(x-x_1)^2}
	\\0&\frac{p_{n+1}(x)-\frac{p_{n+1}(x_1)}{p_n(x_1)}p_n(x)}{x-x_1}
	\end{bmatrix}M^{-1},\\
	\hat P^{[2]}(x)&=-M
	\begin{bmatrix}
	\frac{K_n(x,x_1)}{p_n(x_1)} & -\frac{K_n(x,x_1)}{p_n(x_1)}p_n'(x_1)+
	\frac{1}{p_n(x_1)}\frac{\partial K_n(x,y)}{\partial y}\big|_{y=x_1}\\
	0 & \frac{K_n(x,x_1)}{p_n(x_1)}
	\end{bmatrix}
	M^{-1}\notag
	\end{align}
Observe that the polynomials \begin{align*}
&p_{n+1}(x)-\frac{p_{n+1}(x_1)}{p_n(x_1)}p_n(x),&	&p_{n+1}(x)-\frac{p_{n+1}(x_1)}{p_n(x_1)}p_n(x)-\frac{W(p_n,p_{n+1})(x_1)}{(p_n(x_1))^2}(x-x_1)p_n(x)
\end{align*}
have a zero at $x=x_1$ of order $1$ and $2$, respectively.

\subsection{Examples}

\begin{exa}\rm
In \phantomsection\cite{grupach} the authors define the notion of a classical pair $\{w(x),D\}$, where $w(x)$ is a symmetric matrix valued weight function and $D$ is a second order linear ordinary differential operator.
In that paper  a  weight function is said to be classical if there exists a second order
linear ordinary differential operator $D$ with matrix valued
polynomial coefficients $A_j(t)$, $\deg A_j\leq j$,
of the form
$D=A_2(x)\frac{\d^2}{\d x^2}+A_1(x)\frac{\d}{\d x}+A_0(x)$,
such that
$\prodint{DP,Q}=\prodint{P,DQ}$
for all matrix valued polynomial functions $P(x)$ and $Q(x)$.
Then, the pair $\{w,D\}$ is called a classical pair.
In example 5.1 in \phantomsection\cite{grupach} they present
a family of Jacobi type classical pairs that contains, up to equivalence, all classical pairs of size two where
$w(x)=x^\alpha(1-x)^\beta F(x)$, with $\alpha,\beta>-1$ and $0<x<1$, and such that $F(x)$ is of degree one and which are irreducible (in the sense that they are not equivalent to a direct sum of classical pairs of size one).
As we will show they are a direct sum of orthogonal polynomials of size 1 produced  by two degree one Christoffel transformations of the scalar Jacobi polynomials with zeros at $x=0,1$. Thus, we are faced with two scalar monic Jacobi polynomials with each of the two parameters $\alpha$ and $\beta$ shifted by one, respectively.  In \phantomsection\cite{tir} an analysis of the reducibility of matrix weights is given. In particular,  in Example 2.4 they consider the case $\alpha=\beta$.  We must stress  that, as was pointed  in \cite{grupach}, reducibility of the matrix of weights $w(x)$ do not imply the reducibility of the classical pair $\{w(x), D\}$. Indeed, despite that the matrix of weights in this example is reducible the corresponding second order linear differential operator is not.

  The classical pair $\{w(x)=x^\alpha(1-x)^\beta F(x),D\}$    is given by
\begin{align*}
F(x)&=F_1x+F_0 ,&
F_1&=\begin{bmatrix}
0&-a\\\
-a&\frac{\beta-\alpha}{\alpha+1}a
\end{bmatrix}, &F_0&=
\begin{bmatrix}
1&1\\
1&1
\end{bmatrix},&
 \text{ with } a&=\frac{\alpha+\beta+2}{\alpha+1},
\end{align*}
and a second order matrix linear ordinary differential operator
\begin{align*}
D=x(1-x)\frac{\d^2}{\d x^2}+(X-xU)\frac{\d}{\d x}+V
\end{align*}
where $U,V,X$  are constant matrices depending on a parameter $u$.
The sequence of orthogonal polynomials $\big\{\tilde{\hat P}^{(\alpha,\beta)}_n(x)\big\}_{n=0}^\infty$ associated with the classical pair is not given in \phantomsection\cite{grupach}. Here  an  explicit representation of  $\tilde{\hat P}^{(\alpha,\beta)}_n(x)$ using Darboux transformations is deduced. In order to do it we consider  that we have an initial alternative Jacobi measure $\d\mu(x)=x^\alpha(1-x)^\beta I_2$, with $\alpha,\beta>-1$ and $0<x<1$, which is perturbed by a degree one matrix polynomial $F$. This matrix polynomial is not  monic but its leading coefficient is non singular and we can write
 \begin{align*}
 F(x)&=F_1 W(x), & W(x)&= I_2x-A, &
A&:=- F_1^{-1}F_0=
\frac{1}{a}\begin{bmatrix}
 \frac{\beta+1}{\alpha+1}&\frac{\beta+1}{\alpha+1}\\
 1& 1
 \end{bmatrix},
 \end{align*}
in terms of a degree one monic matrix polynomial $W(x)$.
We have that $A$ has two different eigenvalues $\sigma(A)=\{0,1\}$ with corresponding eigenvectors
$[1,-1]^\top$ and $\Big[\frac{\beta+1}{\alpha+1},1\Big]^\top$, the matrix $M:=\begin{bsmallmatrix}
1 & \frac{\beta+1}{\alpha+1}\\
-1 & 1
\end{bsmallmatrix}$
allows to write $A=M\diag(0,1)M^{-1}$.

Remember, as was noticed in Remark  \ref{remark:non_monic},  that  from the monic orthogonal polynomials $\hat P_{n}^{(\alpha,\beta),[1]}(x)$ with respect to $W$, we get
\begin{align*}
\tilde {\hat P}_n^{(\alpha,\beta)}(x)=F_1 \hat P_{n}^{(\alpha,\beta),[1]}(x) F_1^{-1},
\end{align*}
which are the monic orthogonal polynomials with respect to $w(x)$.
As the matrix of measures $F(x)\d \mu(x)$ is symmetric, the  bi-orthogonality collapses to orthogonality and the super-indexes $[1]$ and $[2]$ can be omitted. We will do the same with    $\hat P_{n}^{(\alpha,\beta),[1]}=\hat P_{n}^{(\alpha,\beta)}$.

 Following \phantomsection\cite{Chel, Chi} we conclude that the set of monic matrix orthogonal polynomials $\{P_n^{(\alpha,\beta)}(x)\}_{n=0}^\infty$ with respect to $\d\mu(x)$  are $P_n^{(\alpha,\beta)}(x)=p_n^{(\alpha,\beta)}(x)I_2$,\footnote{We must be careful at this point with the notation. This is not the scalar standard Jacobi polynomial usually denoted by the same symbol. In fact, if $\mathcal P^{(\alpha,\beta)}(z)$ denotes the standard Jacobi polynomials, then $p^{(\alpha,\beta)}(x)=\frac{2^n}{S_n(\alpha,\beta)}\mathcal P^{(\beta,\alpha)}(2x-1)$, notice the interchange between the parameters $\alpha\rightleftharpoons\beta$ and the linear transformation of the independent variable $x$.} with the alternative Jacobi polynomials $p_{n}^{(\alpha,\beta)}(x)$ given by
\begin{align*}
p_{n}^{(\alpha,\beta)}(x)&=\frac{1}{S_n(\alpha,\beta)}\sum_{k=0}^n\binom{n+\alpha}{n-k}\binom{n+\beta}{k}x^{n-k}(x-1)^k,&
&\text{with } &S_n(\alpha,\beta)&=\binom{2n+\beta+\alpha}{n}.
\end{align*}
We easily see that
\begin{align*}
p^{(\alpha,\beta)}_n(0)&=\frac{1}{S_n(\alpha,\beta)}\binom{n+\alpha}{n},& p^{(\alpha,\beta)}_n(1)&=\frac{1}{S_n(\alpha,\beta)}\binom{n+\beta}{n},
\end{align*}
so that
\begin{align*}
\frac{p^{(\alpha,\beta)}_{n+1}(0)}{p^{(\alpha,\beta)}_{n}(0)} 
&
=(n+1+\alpha)\rho^{(\alpha,\beta)}_n,&
\frac{p^{(\alpha,\beta)}_{n+1}(1)}{p^{(\alpha,\beta)}_{n}(1)} &    =(n+1+\beta)\rho^{(\alpha,\beta)}_n,
\end{align*}
where
\begin{align*}
\rho^{(\alpha,\beta)}_n&:=\frac{(n+\beta+\alpha+1)}{(2n+\beta+\alpha+2)(2n+\beta+\alpha+1)}.
\end{align*}

From \eqref{eq:darboux_diagonal} we conclude
\begin{align*}
\hat P_n^{(\alpha,\beta)}(x)
&=
M\begin{bmatrix}
\frac{p_{n+1}^{(\alpha,\beta)}(x)-(n+1+\alpha)\rho^{(\alpha,\beta)}_np_n^{(\alpha,\beta)}(x)}{x}& 0\\
0 & \frac{p_{n+1}^{(\alpha,\beta)}(x)-(n+1+\beta)\rho^{(\alpha,\beta)}_np_n^{(\alpha,\beta)}(x)}{x-1}
\end{bmatrix}M^{-1}.
\end{align*}
However, we must notice that these two Darboux transformations correspond to the following transformations of the Jacobi measure
\begin{align*}
x^\alpha(x-1)^\beta\mapsto x(x^\alpha(x-1)^\beta)&=x^{\alpha+1}(x-1)^\beta,&x^\alpha(x-1)^\beta\mapsto (x-1)(x^\alpha(x-1)^\beta)=x^{\alpha}(x-1)^{\beta+1},
\end{align*}
i.e. the transformations correspond to the shifts $\alpha\mapsto \alpha+1$  and $\beta\mapsto\beta+1$,  respectively. Consequently,
\begin{align*}
\hat P_n^{(\alpha,\beta)}(x)
&=
M\begin{bmatrix}
p_{n}^{(\alpha+1,\beta)}(x)
& 0\\
0 &
p_{n}^{(\alpha,\beta+1)}(x)\end{bmatrix}M^{-1}.
\end{align*}

With the matrix
\begin{align*}
\tilde M:=\begin{bmatrix}
-1 & 1\\
\frac{1+\beta}{1+\alpha} & 1
\end{bmatrix}
\end{align*}
we can write $F_1M=-a\tilde M$. We finally get  the monic matrix orthogonal polynomials
\begin{align*}
\tilde{\hat P}_n^{(\alpha,\beta)}(x)=
\tilde M\begin{bmatrix}
p_{n}^{(\alpha+1,\beta)}(x)
& 0\\
0 &
p_{n}^{(\alpha,\beta+1)}(x)\end{bmatrix}\tilde M^{-1},
\end{align*}
 for the matrix of measures $\tilde W(x)\d\mu(x)$ in example 5.1 of \phantomsection\cite{grupach} which are explicitly expressed in terms of scalar Jacobi polynomials as follows
\begin{align*}
\\
\tilde{\hat P}_n^{(\alpha,\beta)}(x)&=\frac{1}{2+\alpha+\beta}\begin{bmatrix}
(\alpha+1)p_{n}^{(\alpha+1,\beta)}(x)+(\beta+1)p_{n}^{(\alpha,\beta+1)}(x)& -(\alpha+1)(p_{n}^{(\alpha+1,\beta)}(x)-p_{n}^{(\alpha,\beta+1)}(x))\\
-(\beta+1)(p_{n}^{(\alpha+1,\beta)}(x)-p_{n}^{(\alpha,\beta+1)}(x)) & (\beta+1)p_{n}^{(\alpha+1,\beta)}(x)+(\alpha+1)p_{n}^{(\alpha,\beta+1)}(x)
\end{bmatrix}.
\end{align*}

To conclude with this example let us mention that in \phantomsection\cite{Cas1} it was found that these matrix orthogonal polynomials also obey a first order ordinary differential equation. From our point of view this is just a consequence of a remarkable fact regarding the  Darboux transformations $p^{(\alpha+1,\beta)}(x),p^{(\alpha,\beta+1)}(x)$ of the original alternative Jacobi polynomials. Under the  hypergeometric function description of  the Jacobi polynomials one gets recurrences for  the Jacobi polynomials. In particular, from the Gauss' contiguous relations one gets   the first order differential relations
\begin{align*}
\Big(x\frac{\d}{\d x}+\alpha+1\Big)p_n^{(\alpha+1,\beta)}(x)&=(\alpha+1+n)p_n^{(\alpha,\beta+1)}(x),\\
\Big((x-1)\frac{\d}{\d x}+\beta+1\Big)p_n^{(\alpha,\beta+1)}(x)&=(\beta+1+n)p_n^{(\alpha+1,\beta)}(x).
\end{align*}

This first order linear ordinary differential system can be re-casted as a matrix linear differential equation as follows
\begin{multline}\label{eq:diff_jacobi_matrix}
\bigg(\begin{bmatrix}
0 &x-1 \\
x& 0
\end{bmatrix}\frac{\d }{\d x}+\begin{bmatrix}
a_1& \beta+1 \\
\alpha+1  &a_2
\end{bmatrix}\bigg)\begin{bmatrix}
p_n^{(\alpha+1,\beta)}(x)& 0\\
0 &p_n^{(\alpha,\beta+1)}(x)
\end{bmatrix}\\=
\begin{bmatrix}
p_n^{(\alpha+1,\beta)}(x)& 0\\
0 &p_n^{(\alpha,\beta+1)}(x)
\end{bmatrix}\begin{bmatrix}
a_1 & n+\beta+1\\
n+\alpha+1 & a_2
\end{bmatrix},
\end{multline}
where $a_1,a_2\in\R$. This  equation is invariant under multiplication on the right and on the left hand side by arbitrary diagonal matrices
\begin{multline*}
\begin{bmatrix}
l_1& 0\\
0 & l_2
\end{bmatrix}\bigg(\begin{bmatrix}
0 &x-1 \\
x & 0
\end{bmatrix}\frac{\d }{\d x}+\begin{bmatrix}
a_1& \beta+1 \\
\alpha+1  &a_2
\end{bmatrix}\bigg)\begin{bmatrix}
r_1& 0\\
0 & r_2
\end{bmatrix}\begin{bmatrix}
p_n^{(\alpha+1,\beta)}(x)& 0\\
0 &p_n^{(\alpha,\beta+1)}(x)
\end{bmatrix}\\=
\begin{bmatrix}
p_n^{(\alpha+1,\beta)}(x)& 0\\
0 &p_n^{(\alpha,\beta+1)}(x)
\end{bmatrix}\begin{bmatrix}
l_1& 0\\
0 & l_2
\end{bmatrix}\begin{bmatrix}
a_1 & n+\beta+1\\
n+\alpha+1 & a_2
\end{bmatrix}\begin{bmatrix}
r_1& 0\\
0 & r_2
\end{bmatrix}.
\end{multline*}
After the similarity transformation, $A\mapsto \tilde M^{-\top} A\tilde M^\top$ we find out that the orthogonal polynomial $\tilde{\hat{P}}_n$ satisfies
\begin{align*}
\Big(A_1(x)\frac{\d}{\d x}+A_0\Big)\big(\tilde{\hat{P}}_n\big)^\top=\big(\tilde{\hat{P}}_n\big)^\top \Lambda_n
\end{align*}
where
\begin{align*}
A_1(x)&=\frac{\alpha+1}{\alpha+\beta+2}\bigg(\begin{bmatrix}
-\delta & \delta_-\\
\delta_+& \delta
\end{bmatrix} x+d \begin{bmatrix}
1& 1\\
-1& -1
\end{bmatrix}\bigg),&A_0&=\frac{\alpha+1}{\alpha+\beta+2}\begin{bmatrix}
C_+-\Delta&\frac{\beta+1}{\alpha+1}(C+\Delta_-)\\
C+\Delta_+& C_-+\Delta
\end{bmatrix},\\
\Lambda_n&=\frac{\alpha+1}{\alpha+\beta+2}\begin{bmatrix}
-\delta & \delta_-\\
\delta_+& \delta
\end{bmatrix}n+A_0,
\end{align*}
with $d=l_1r_2$ and
 \begin{align*}
\delta&=l_1r_2
+\frac{\beta+1}{\alpha+1} l_2r_1, &\delta_-&=-l_1r_2+\Big(\frac{\beta+1}{\alpha+1}\Big)^2  l_2r_1, &\delta_+&=l_1r_2- l_2r_1,\\
\Delta&=(\beta+1)(l_1r_2+ l_2r_1), &
\Delta_-&=-l_1r_2(\alpha+1)+l_2r_1(\beta+1),&\Delta_+&=l_1r_2(\beta+1)-l_2r_1(\alpha+1),\\
C&=-l_1r_1a_1+ l_2r_2a_2, &C_-&=\frac{\beta+1}{\alpha+1}l_1r_1a_1+ l_2r_2a_2, &C_+&=l_1r_1a_1+\frac{\beta+1}{\alpha+1} l_2r_2a_2.
\end{align*}

When we take $l_1=l_2=-1$, $r_1=r_2=1$ and $a_1=\beta+1$ and $a_2=\alpha+1$ we get the first order ordinary differential system  in \S 4 of \phantomsection\cite{Cas1}.

\begin{rem}
	The discussion in this example, regarding   the Jacobi polynomials $p^{(\alpha+1,\beta)}(x)$ and $p^{(\alpha,\beta+1)}(x)$ and the use of the Gauss' contiguous  relations, connects with the results in \cite{koelink1}, Remark 2.8., see also \cite{koelink3,koelink4}.
\end{rem}

\end{exa}

\begin{exa}\rm
	Here we analyze  the Chebyshev example  taken from \phantomsection\cite{Cas1}, that gives an example of a family of matrix orthogonal polynomials which  satisfy a first order linear ordinary differential equation.	In \S 3 of \phantomsection\cite{Cas1} we find a set of MOPS related with the measure $\tilde W(x)\d\mu(x)$ where
	\begin{align*}
\tilde	W(x)&:=\begin{bmatrix}
	1 & x\\
	x &1
	\end{bmatrix}, & \d\mu(x)&=\frac{1}{\sqrt{1-x^2}}.
	\end{align*}
	We have a nonsingular leading coefficient $\begin{bsmallmatrix}
0 & 1\\
1 & 0
	\end{bsmallmatrix}$ so that
\begin{align*}
\tilde W(x)&=\begin{bmatrix}
0 & 1\\
1& 0
\end{bmatrix} W(x), & W(x):=I_2 x-\begin{bmatrix}
0 & -1\\
-1& 0
\end{bmatrix}.
\end{align*}
Following Remark \ref{remark:non_monic}	we shall analyze the Darboux transformations $\d\mu(x)\mapsto W(x)\d\mu(x)$.

Thus, using \eqref{eq:darboux_diagonal} we can write  the perturbed monic matrix orthogonal polynomials as follows
\begin{align*}
\hat P_n(x)&=Q\begin{bmatrix}
\frac{t_{n+1}(x)-\frac{t_{n+1}(-1)}{t_n(-1)}t_n(x)}{x+1} & 0\\
0 & \frac{t_{n+1}(x)-\frac{t_{n+1}(1)}{t_n(1)}t_n(x)}{x-1}
\end{bmatrix}Q^\top, & Q&:= \frac{1}{\sqrt 2}\begin{bmatrix}
1 & 1\\
1 & - 1
\end{bmatrix},
\end{align*}
where $\{t_n(x)\}_{n=0}^\infty$ are the monic Chebyshev polynomials of first kind, i.e., $t_n(x)=2^{-n+1}T_n(x)$ with $T_n$ the first kind Chebyshev polynomial of degree $n$. Therefore, recalling that $T_n(\pm 1)=(\pm 1)^n$ we get
\begin{align*}
\frac{t_{n+1}(\mp 1)}{t_n(\mp 1)}& =\mp\frac{ 1}{2}, &
t_{n+1}(x)-\frac{t_{n+1}(\mp 1)}{t_n(\mp 1)}t_n(x)&=\frac{1}{2^n}(T_{n+1}(x)\pm T_n(x)).
\end{align*}
Now, recalling the mutual recurrence relation satisfied by  Chebyshev polynomials of the  first and second kind, denoted these last ones by $U_n$,
\begin{align*}
T_{n+1}(x) &= xT_n(x) - (1 - x^2)U_{n-1}(x), &
T_n(x) &= U_n(x) - x U_{n-1}(x),
\end{align*}
which implies $	T_{n+1}(x) = xU_n(x) -  U_{n-1}(x)$, we deduce
\begin{align*}
T_{n+1}(x)\pm T_n(x)=(x\pm 1) (U_n(x)\mp U_{n-1}(x)).
\end{align*}
Consequently,	
	\begin{align*}
	\hat P_n(x)&=\frac{1}{2^n}Q\begin{bmatrix}
U_n(x)- U_{n-1}(x)& 0\\
	0 & U_n(x)+ U_{n-1}(x)
	\end{bmatrix}Q^\top, & Q&:= \frac{1}{\sqrt 2}\begin{bmatrix}
	1 & 1\\
	1 & - 1
	\end{bmatrix}.
	\end{align*}
\end{exa}
The matrix orthogonal polynomials associated with the original measure $\tilde W(x)\d\mu(x)$ can be recovered from this by a similarity transformation with $\begin{bsmallmatrix}
0 & 1\\
1&0
\end{bsmallmatrix}$ so that
	\begin{align*}
	\tilde{\hat P}_n(x)&=\frac{1}{2^{n+1}}\begin{bmatrix}
	1 & -1 \\
	1 & 1
	\end{bmatrix}\begin{bmatrix}
	U_n(x)- U_{n-1}(x)& 0\\
	0 & U_n(x)+ U_{n-1}(x)
	\end{bmatrix}\begin{bmatrix}
	1 & 1 \\
	-1 & 1
	\end{bmatrix}\\
	&=\frac{1}{2^{n}}\begin{bmatrix}
U_n(x) & -U_{n-1}(x)\\
-U_{n-1}(x) & U_n(x)
	\end{bmatrix}.
	\end{align*}
\begin{rem}
The  polynomials $\{U_n\mp U_{n-1}\}_{n=0}^\infty$ with $U_{-1}=0$, which are orthogonal with respect to the measures $\frac{x\pm 1}{\sqrt{1-x^2}}$,  are the well known Chebyshev polynomials of the third and fourth kind, respectively.
\end{rem}
\begin{rem}
	The symmetric structure of the MOPS can be encoded in the equation
	\begin{align*}
\begin{bmatrix}
0 & 1\\
1 & 0
\end{bmatrix} \tilde{\hat P}_n(x)=\tilde{\hat P}_n(x)\begin{bmatrix}
0 & 1\\
1 & 0
\end{bmatrix}.
	\end{align*}
\end{rem}

As in the Jacobi case, the new two scalar families of orthogonal polynomials are related through
	\begin{align}\label{eq:chebysev}
\Big(	(x\mp 1)\frac{\d}{\d x}+\frac{1}{2}\Big)\big(U_n(x)\pm U_{n-1}(x)\big)=(n+\frac{1}{2})\big(U_n(x)\mp U_{n-1}(x)\big).
	\end{align}
This follows from
\begin{align*}
(x\mp 1) (U'_n(x)\pm U'_{n-1}(x))&= (x\mp 1)\frac{\d }{\d x}\Big(\frac{T_{n+1}(x)\mp T_n(x)}{x\mp 1}\Big)
\\
&=T'_{n+1}(x)\mp T'(x)-\frac{T_{n+1}(x)\mp T_n(x)}{x\mp 1}\\
&=(n+1)U_n(x)-n U_{n-1}(x)-U_n(x)\mp U_{n-1}(x).
\end{align*}
Here we have used that $T_n'=nU_{n-1}$.

Differential equation \eqref{eq:chebysev} can be written in matrix form
\begin{multline}\label{eq:diff_chesybev_matrix}
\bigg(\begin{bmatrix}
0 &x-1 \\
x+1 & 0
\end{bmatrix}\frac{\d }{\d x}+\begin{bmatrix}
a_1& \frac{1}{2} \\
\frac{1}{2}  &a_2
\end{bmatrix}\bigg)\begin{bmatrix}
U_n(x)- U_{n-1}(x)& 0\\
0 & U_n(x)+ U_{n-1}(x)
\end{bmatrix}\\=
\begin{bmatrix}
U_n(x)- U_{n-1}(x)& 0\\
0 & U_n(x)+ U_{n-1}(x)
\end{bmatrix}\begin{bmatrix}
a_1 & n+\frac{1}{2} \\
n+\frac{1}{2}  & a_2
\end{bmatrix}
\end{multline}
where $a_1,a_2\in\R$ are arbitrary constants. Notice also that this matrix equation is invariant under multiplication on the right and on the left hand sides by arbitrary diagonal matrices $L=\diag(l_1,l_2)$ and $R=\diag(r_1,r_2)$,
\begin{multline*}
\begin{bmatrix}
l_1& 0\\
0 & l_2
\end{bmatrix}\bigg(\begin{bmatrix}
0 &x-1 \\
x+1 & 0
\end{bmatrix}\frac{\d }{\d x}+\begin{bmatrix}
a_1& \frac{1}{2} \\
\frac{1}{2}  &a_2
\end{bmatrix}\bigg)\begin{bmatrix}
r_1& 0\\
0 & r_2
\end{bmatrix}\begin{bmatrix}
U_n(x)- U_{n-1}(x)& 0\\
0 & U_n(x)+ U_{n-1}(x)
\end{bmatrix}\\=
\begin{bmatrix}
U_n(x)- U_{n-1}(x)& 0\\
0 & U_n(x)+ U_{n-1}(x)
\end{bmatrix}\begin{bmatrix}
l_1& 0\\
0 &l_2
\end{bmatrix}\begin{bmatrix}
a & n+\frac{1}{2} \\
n+\frac{1}{2} & b
\end{bmatrix}\begin{bmatrix}
r_1& 0\\
0 & r_2
\end{bmatrix}.
\end{multline*}
After the similarity transformation we find out that the orthogonal polynomial $\tilde{\hat{P}}_n$ satisfies
\begin{align*}
\Big(A_1(x)\frac{\d}{\d x}+A_0\Big)\tilde{\hat{P}}_n(x)=\tilde{\hat{P}}_n(x) \Lambda_n,
\end{align*}
where
\begin{align*}
A_1(x)&=\begin{bmatrix}
-\delta_+x+\delta_- & \delta_-x-\delta_+\\
-\delta_-x+\delta_+ & \delta_+x-\delta_-
\end{bmatrix},&
\Lambda_n&=\begin{bmatrix}
A_+-\delta_+\Big(n+\frac{1}{2}\Big) &A_-+\delta_-\Big(n+\frac{1}{2}\Big)\\
A_--\delta_-\Big(n+\frac{1}{2}\Big) & A_++\delta_+\Big(n+\frac{1}{2}\Big)
\end{bmatrix},
\end{align*}
and $A_0=\Lambda_{n=0}$
with
\begin{align*}
\delta_{\pm}&=l_1r_2\pm l_2r_1, &
A_\pm&=l_1r_1a_1\pm l_2r_2a_2.
\end{align*}
Equations (3.1) and (3.2) of \S 3 of \phantomsection\cite{Cas1} can be recovered choosing $\Big(\delta_+=A_+=0,\delta_-=1,A_-=-\frac{1}{2}\Big)$
and $\Big(\delta_-=A_-=0,\delta_+=-1,A_+=\frac{1}{2}\Big)$, respectively.

However, they are all equivalent to \eqref{eq:diff_chesybev_matrix}, another form of writing \eqref{eq:chebysev}.
It is in fact this last equation  \eqref{eq:chebysev} a quite interesting  one. Indeed, we have two families of Darboux transformed orthogonal polynomials interconnected by two first order differential equations.
Moreover, we conclude
\begin{align*}
\Big((x^2-1)\frac{\d^2}{\d x^2}+(2x\mp 1)\frac{\d}{\d x}+\frac{1}{4}\Big)\big(U_n(x)\mp U_{n-1}(x)\big)=\Big(n+\frac{1}{2}\Big)^2\big(U_n(x)\mp U_{n-1}(x)\big).
\end{align*}

\paragraph{\textbf{Example 3}} Here we comment on the matrix Gegenbauer matrix valued polynomials  discussed in  \cite{koelink1}. In this case the matrix of weights is a symmetric matrix, $W^{(\nu)}:[-1,1]\to\R^{N\times N}$, with matrix coefficients  of the form 
\begin{align*}
(W^{(\nu)}(x))_{i,j}&:=(1-x^2)^{\nu-1/2}\sum_{k=\operatorname{max}(0,i+j+1-N)}\alpha^{(\nu)}_k(i,j)C^{(\nu)}_{i+j-2k}(x), & i\geq j,
\end{align*}
where  $\alpha^{(\nu)}_t(i,j)$ are some coefficients and $C^{(\nu)}_n(x)$ stands for the Gegenbauer or ultraspherical polynomials. Erik Koelink and Pablo Román kindly communicated us a nice feature of the matrix Christoffel transformation discussed in this paper when acting on this reach family of MOPS: two families of Gegenbauer MOPS  associated with matrices of weights $W^{(\nu_1)}(x)$ and $W^{(\nu_2)}(x)$,
such that $\nu_1-\nu_2=m\in\Z$, are linked by a matrix Christoffel transformation. Now, the perturbing polynomial $W(x)$ has $\deg W=2m$. These examples are, in general, reducible to two irreducible blocks of sizes $N/2$, for $N$ even, and $(N+1)/2$ and $(N-1)/2$ for odd $N$. For a  discussion on the orthogonal and non orthogonal reducibility of these examples see \cite{koelink1,koelink2}.

\section{Singular leading coefficient  matrix polynomial perturbations}

After studying some examples that the literature provides us with, one may realize that, even thought it is generic to assume the perturbing matrix polynomial ${W}(x)$ to have a nonsingular leading coefficient, many examples do have a singular matrix as its leading coefficient.
This situation is a special feature of the matrix case setting since in the scalar case,
having a singular leading term would mean that this coefficient is just zero (affecting, of course, to the degree of the polynomial).
For this reason, when dealing with this kind of matrix polynomials talking about their degree should make no sense.
The effect that this fact has on our reasoning is that since $\deg[\det {W}(x)]\leq Np$ the information encoded in the zeroes
(and corresponding adapted polynomials) of $\det {W}(x)$ is no longer enough to make the matrices $\Pi_{kN}$ of the needed size.
Therefore, there will be no way to express the perturbed polynomials just in terms of the initial ones evaluated at the zeroes of
$\det {W}(x)$ and the method to find a Christoffel type formula fails. However, the information that seems to be missing in these cases may actually not be necessary due to
the singular character of the leading coefficient of the perturbing polynomial. Let us consider the following  example to take a glimpse of this scenario.

Let us pick up some scalar measure $\d \mu(x)$ and its associated monic OPs $\big\{p_k(x)\big\}_{k=0}^\infty$ together with
their norms and three term recurrence relation
\begin{align*}
h_k \delta_{kj}&:=\langle p_k,p_j \rangle,  &
xp_k(x)&=J_{k,k-1}p_{k-1}(x)+J_{k,k}p_k(x)+p_{k+1}(x), &
J_{k,k-1}&=\frac{h_k}{h_{k-1}}>0.
\end{align*}
Now, consider its $2q\times 2q$ matrix diagonal extension $\in \mathbb{R}^{2q\times 2q}[x]$
\begin{align*}
P_k(x)&:=p_k(x) I_{2q},  &  H_k&:=h_k I_{2q}.
\end{align*}
Our aim is to consider the following matrix polynomial (with singular leading coefficient)
\begin{align*}
{W}(x)&:=\begin{bmatrix}
I_q+AA^\top x^2 & Ax \\
A^{\top}x     & I_q
\end{bmatrix}, &A\in\R^{q \times q},
\end{align*}
which is inspired by the $q=1$ case
$\begin{bsmallmatrix}
1+a^2x^2 & ax \\
ax     & 1
\end{bsmallmatrix} $  (see \phantomsection\cite{Dur6}, and references therein)
and study the corresponding perturbations of our initial scalar measure; i.e.,
$\d \hat{\mu}(x):={W}(x)\d \mu(x)$
in order to obtain the transformed matrix orthogonal polynomials
\begin{align*}
\hat{P}_{k}(x)&:=\begin{bmatrix}
(\hat{P}_{k})_{1,1} & (\hat{P}_{k})_{1,2} \\
(\hat{P}_{k})_{2,1} & (\hat{P}_{k})_{2,2}
\end{bmatrix},  &
&\hat{P}_{k}(x)\in \mathbb{R}^{2q\times 2q}[x], (\hat{P}_{k})_{i,j}\in \mathbb{R}^{q\times q}[x], \\
\langle \hat{P}_{k},\hat{P}_{j} \rangle_{W}&:=
\delta_{k,j}\hat{H}_{k}=\delta_{k,j}\begin{bmatrix}
(\hat{H}_{k})_{1,1} & (\hat{H}_{k})_{1,2} \\
(\hat{H}_{k})_{2,1} & (\hat{H}_{k})_{2,2}
\end{bmatrix}, &
&\hat{H}_{k}\in \mathbb{R}^{2q\times 2q},(\hat{H}_{k})_{i,j}\in \mathbb{R}^{q\times q}.
\end{align*}
We have splitted them up this way for computational purposes.
Notice that since ${W}(x)={W}(x)^{\top}$ we have
$\hat{M}=\hat{M}^{\top}:=\hat{S}^{-1}\hat{H}[\hat{S}^{-1}]^{\top}$ and, therefore,
$\hat{P}^{[1]}=\hat{P}^{[2]}:=\hat{P}$ and $\hat{H}_{k}=(\hat{H}_{k})^{\top}$.

Let us point out that
\begin{align*}
{W}(x)&=\mathcal W(x)\mathcal W(x)^{\top}, &
\mathcal W&:=\begin{bmatrix}
I_q & Ax \\
0 & I_q
\end{bmatrix}, &
\mathcal W^{-1}=\begin{bmatrix}
I_q & -Ax \\
0 & I_q
\end{bmatrix}.
\end{align*}
This implies that $\det W=\det \mathcal W=1$ and, consequently, there is no spectral analysis to perform as there are non eigenvalues at all.
Thus, the relation between the original  and perturbed measures and moment matrices is
\begin{align*}
[\mathcal W(x)]^{-1} \d \hat{\mu} &= \d \mu [\mathcal W(x)]^{\top}, &
[\mathcal W(\Lambda)]^{-1} \hat{M} &= M [\mathcal W (\Lambda)]^{\top}.
\end{align*}
\begin{defi}\label{def:omega_duran}
	We introduce the resolvent or connection matrix
	\begin{align*}
	\omega:=\hat{S}\mathcal W(\Lambda)S^{-1}.
	\end{align*}
\end{defi}
\begin{pro}\label{pro:omega_duran}
	The matrix $\omega$ is block tridiagonal, having only its diagonal and first superdiagonal and subdiagonal nonzero, and satisfies
	\begin{align*}
	\omega^{-1}&=H [\omega]^{\top} \hat{H}^{-1}.
	\end{align*}
	Moreover, we have the important connection formula
	\begin{align*}
	\omega P=\hat{P} \mathcal W(x).
	\end{align*}
\end{pro}
\begin{proof}
	The first relation is a consequence of the $LU$ factorization of the moment matrices and the connection formula is a straightforward consequence of the  definition of $\omega$.
\end{proof}

\begin{pro}
\begin{enumerate}
	\item The matrices
	\begin{align*}
	\rho_{k+1}&:=\left(I_q+J_{k+1,k}A^{\top}A\right)^{-1}, & k&\in\{-1,0,1,\dots \},
	\end{align*}
	exist.
	\item
	The perturbed MOPS can be writen in terms of the original OPS as follows
	\begin{align*}
	\hat{P}_{1,k+1}(x)\mathcal W(x)&=-\begin{bmatrix}
	J_{k+1,k}J_{k+1,k+1}A \rho_{k+1} A^{\top} & 0 \\
	J_{k+1,k}\rho_{k+1} A^{\top}  & 0
	\end{bmatrix} p_{k}(x)
	+\begin{bmatrix}
	I_{q} & J_{k+1,k+1}A\rho_{k+1} \\
	0 & \rho_{k+1}
	\end{bmatrix}p_{k+1}(x)+
	\begin{bmatrix}
	0 & A \\
	0 & 0
	\end{bmatrix}p_{k+2}(x),
	\end{align*}
	for $k\in\{-1,0,1,\dots \}.$
	\end{enumerate}
\end{pro}
\begin{proof}
	From the $(k+1)$-th row of the connection formula we have that
	\begin{align*}
	\omega_{k+1,k}p_k(x)+\omega_{k+1,k+1}p_{k+1}(x)+\omega_{k+1,k+2}p_{k+2}(x)=\hat{P}_{k+1}(x)\mathcal{W}(x),
	\end{align*}
	but from the Definition \ref{def:omega_duran} and Proposition \ref{pro:omega_duran} one realizes that the previous expression reads
	\begin{align*}
	\hat{P}_{k+1}(x)\mathcal{W}(x)=	\hat{H}_{k+1}\begin{bmatrix}
	0 & -A \\
	0 & 0
	\end{bmatrix}^{\top} h_{k}^{-1}p_k(x)+
	\begin{bmatrix}
	(\omega_{k+1,k+2})_{11} & (\omega_{k+1,k+2})_{12} \\
	(\omega_{k+1,k+2})_{21} & (\omega_{k+1,k+2})_{22}
	\end{bmatrix}p_{k+1}(x)+
	\begin{bmatrix}
	0 & A \\
	0 & 0
	\end{bmatrix}p_{k+2}(x).
	\end{align*}
	Now, taking into account that both $(\hat{P}_{k+1})_{11},(\hat{P}_{k+1})_{22}$ are monic $q\times q$ polynomials of degree $k+1$, while
	$(\hat{P}_{k+1})_{12},(\hat{P}_{k+1})_{21}$ are $q\times q$ polynomials of degree less than $k+1$, it is not hard to see
	(after using the three term recurrence relation of the initial polynomials) that
	\begin{align*}
	(\omega_{k+1,k+2})_{11}&=I_q, & (\omega_{k+1,k+2})_{12}&=J_{k+1,k+1}A-h_{k}^{-1}(\hat{H}_{k+1})_{12}A^{\top}A, \\
	(\omega_{k+1,k+2})_{21}&=0, & (\omega_{k+1,k+2})_{22}&=I_q-h_{k}^{-1}(\hat{H}_{k+1})_{22}A^{\top}A.
	\end{align*}
	Hence, we have every coefficient that appears in the connection formula in terms of the still unknown norms of the MOPS.
	Therefore, we just need to compute the second block column of the following integral
	\begin{multline*}
	\int \left[\omega_{k+1,k}p_k(x)+\omega_{k+1,k+1}p_{k+1}(x)+\omega_{k+1,k+2}p_{k+2}(x)\right]\left[(\mathcal{W}(x))^{\top}x^{k+1}\right] \d \mu(x) \\\begin{aligned}
	&=\int \hat{P}_{k+1}(x)\mathcal{W}(x)\left[(\mathcal{W}(x))^{\top} x^{k+1}\right] \d \mu(x)\\&= \int \hat{P}_{k+1}(x) \d \hat{\mu}(x) x^{k+1} \\&= \hat{H}_{k+1},
	\end{aligned}
	\end{multline*}
	which yields
	\begin{align*}
	(\hat{H}_{k+1})_{12}&=J_{k+1,k+1}h_{k+1}A \left(I_q+J_{k+1,k}A^{\top}A\right)^{-1} ,
	&  (\hat{H}_{k+1})_{22}&=h_{k+1}\left(I_q+J_{k+1,k}A^{\top}A\right)^{-1}.
	\end{align*}
	Therefore,
	\begin{align*}
	(\omega_{k+1,k+2})_{12}&=J_{k+1,k+1}A \left(I_q+J_{k+1,k}A^{\top}A\right)^{-1}, &
	(\omega_{k+1,k+2})_{22}&=\left(I_q+J_{k+1,k}A^{\top}A\right)^{-1},
	\end{align*}
	and the stated result follows.
\end{proof}
For $q=1$ and the classical measures we have, see \cite{duran-lopez},
\begin{coro} Starting from the classical measures
	\begin{enumerate}
		\item Hermite monic polynomials $\big\{\mathcal{H}_k(x)\big\}_{k=0}^\infty$ with norm $h_k=\pi^{\frac{1}{2}}\frac{k!}{2^k}$
		\begin{align*}
		J_{k+1,k}&=\frac{k+1}{2}, & J_{k+1,k+1}&=0,    & \rho_{k+1}:=\frac{2}{2+a^2(k+1)}.
		\end{align*}
		\item Laguerre monic polynomials $\big\{\mathcal{L}^{\alpha}_k(x)\big\}_{k=0}^\infty$ with norm $h_k=k!\Gamma(k+1+\alpha)$
		\begin{align*}
		J_{k+1,k}&=(k+1)(k+\alpha+2),  &  J_{k+1,k+1}&=(2k+\alpha+3), & \rho_{k+1}:=\frac{1}{1+a^2(k+1)(k+1+\alpha)}.
		\end{align*}
	\end{enumerate}
	and perturbing them by the matrix polynomial
	\begin{align*}
	{W}(x)&=\mathcal{W}(x)\left(\mathcal{W}(x)\right)^{\top}, &  \mathcal{W}(x)=\begin{pmatrix}
	1 &  ax \\
	0 &   1
	\end{pmatrix},
	\end{align*}
	one obtains the  perturbed MOPS related to the classical OPS as  follows
	\begin{align*}
	\hat{\mathcal{H}}_{k+1}(x)\mathcal{W}(x)&=\begin{pmatrix}
	0 & 0 \\
	\frac{\rho_{k+1}-1}{a} & 0
	\end{pmatrix} \mathcal{H}_k(x)+
	\begin{pmatrix}
	1 & 0 \\
	0 & \rho_{k+1}
	\end{pmatrix}\mathcal{H}_{k+1}(x)+
	\begin{pmatrix}
	0 & a \\
	0 & 0
	\end{pmatrix}\mathcal{H}_{k+2}(x), \\
	\hat{\mathcal{L}}^{\alpha}_{k+1}(x)\mathcal{W}(x)&=\begin{pmatrix}
	(\rho_{k+1}-1)(2k+3+\alpha) & 0 \\
	\frac{\rho_{k+1}-1}{a} & 0
	\end{pmatrix}\mathcal{L}^{\alpha}_k(x)+
	\begin{pmatrix}
	1 & a(2k+3+\alpha)\rho_{k+1} \\
	0 & \rho_{k+1}
	\end{pmatrix}\mathcal{L}^{\alpha}_{k+1}(x)+
	\begin{pmatrix}
	0 & a \\
	0 & 0
	\end{pmatrix}\mathcal{L}^{\alpha}_{k+2}(x).
	\end{align*}
\end{coro}

\section{Extension to non-Abelian 2D Toda hierarchies}
Matrix orthogonal polynomials are connected with non-Abelian Toda lattices, see \phantomsection\cite{mir,amu}.
\subsection{Block Hankel moment matrices vs multi-component Toda hierarchies}
Let us take $M=(m_{i,j})_{i,j=0}^\infty$, $m_{i,j}\in\R^{p\times p}$  a semi-infinite block matrix having  a Gaussian factorization
\begin{align*}
M=(S_1)^{-1} H (S_2)^{-\top},
\end{align*}
where $S_1,S_2$ are lower uni-triangular block matrices and $H$ is  block diagonal. Notice that conditions for this factorization to hold were given in Proposition \ref{qd1}.

\begin{defi}
	 We  introduce some continuous  flows or perturbations of this semi-infinite matrix. For that aim we
first  consider the diagonal matrices
\begin{align*}
t_{i,j}=&\diag(t_{i,j,1},\dots,t_{i,j,p})\in\R^{p\times p}, & i&=1,2, &j&\in\Z_+,
\end{align*}
the semi-infinite undressed wave matrices
\begin{align*}
V_i^{(0)}&:=\exp\Big(\sum_{j=0}^\infty t_{i,j}\Lambda^j\Big), & i&=1,2,
\end{align*}
and the perturbed matrix $M(t)$, $t=(t_1,t_2)$, $t_i=\{t_{i,j,a}\}_{\substack{j\in\Z_+\\ a\in\{1,\dots,p\}}}$
\begin{align*}
M(t)=& V_1^{(0)}(t_1)M\big(V_2^{(0)}(t_2)\big)^{-\top}.
\end{align*}
\end{defi}
Observe that we do not require any Hankel form for the matrix $M$, modelled by  $\Lambda M=M\Lambda ^\top$.
However, if $M(0)$ is a Hankel matrix  $M(t)$ is also a Hankel matrix taking into account $\Lambda M(t)= M(t) \Lambda ^{\top}$. Hence, if $\d\mu(x)$ is the initial matrix of measures, then the new matrix of measures $\d\mu(x,t)$ will be
\begin{align*}
\d\mu(x,t)=\exp\Big(\sum_{j=0}^\infty  t_{1,j}x^j\Big)\d\mu(x)\exp\Big(-\sum_{j=0}^\infty t_{2,j}x^j\Big).
\end{align*}
Here $M(t)$ will be the moment matrix of the matrix of measures.
Moreover, if at any time the matrix of measures is block Hankel then it was and it will be a  Hankel block matrix at any time.
If we assume that we can perform the Gaussian factorization again we can write
\begin{align*}
M(t)&=(S_1(t))^{-1} H(t) (S_2(t))^{-\top}.
\end{align*}
As we know, for the block Hankel case we are dealing with bi-orthogonal or orthogonal polynomials with respect to the associated matrix of measures. What happens in the general case?
Following \phantomsection\cite{adler-van moerbeke}
and \phantomsection\cite{manas0} we can understand the Gaussian factorization also as a bi-orthogonality condition.
The semi-infinite vectors of polynomials will be
\begin{align*}
P^{[1]}(x)&:=S_1(t)\chi(x), & P^{[2]}(x)&:= S_2(t)\chi(x),
\end{align*}
and we consider a sesquilinear form in $\R^{p\times p}[x]$, see \S \ref{S:mop}, that for any couple of matrix polynomials $P=\sum_{k=0}^{\deg P}p_kx^k$ and $Q(x)=\sum_{l=0}^{\deg Q} q_lx^l$ is defined by
\begin{align*}
\prodint{P(x),Q(x)}=\sum_{\substack{k=1,\dots,\deg P\\
		l=1,\dots,\deg Q}}p_k M_{k,l}(q_l)^\top,
\end{align*}
where we can interpret
\begin{align*}
M_{k,l}=\prodint{x^k 1_p,x^l 1_p}
\end{align*}
as the Gram matrix of the sesquilinear form.
With respect to this sesquilinear form we have the bi-orthogonality condition
\begin{align*}
\prodint{P_k^{[1]}(x),P_l^{[2]}(x)}=H_k\delta_{k,l}.
\end{align*}

For a block Hankel initial condition  this sesquilinear form is just the sequilinear product associated with a linear functional of a measure. In \phantomsection\cite{manas3} different examples are discussed for the matrix orthogonal polynomials scenario. For example, multigraded Hankel matrices $M$ fulfilling
\begin{align*}
\Big(\sum_{a=1}^p \Lambda^{n_a}E_{a,a}\Big)M=M\Big(\sum_{a=1}^p \big(\Lambda^\top\big)^{m_a}E_{a,a}\Big),
\end{align*}
where $n_1,\dots,n_p,m_1,\dots,m_p$ are positive integers,
can be realized as
\begin{align*}
M_{k,l}=\int x^k\d\mu^{(l)}(x)
\end{align*}
in terms of  matrices of measures $\d\mu^{(l)}(x)$ which satisfy the following periodicity condition
\begin{align}\label{eq: matrix of measures periodicity}
\d\mu^{(l+m_a)}_{a,b}(x)=x^{n_a}\d\mu^{(l)}_{a,b}(x).
\end{align}
Therefore, given the measures $\d\mu^{(0)}_{a,b},\dots,\d\mu^{(m_b-1)}_{a,b}$ we can recover all the others from \eqref{eq: matrix of measures periodicity}.
	In this case, we have generalized orthogonality conditions like
\begin{align*}
\int P^{[1]}_k(x) \d \mu^{(l)}(x) &=0, & l=0,\dots, k-1.
\end{align*}
	
Coming back to the Gaussian factorization, we consider the wave matrices
\begin{align*}
V_1(t):=&S_1(t)V_1^{(0)}(t_1), \\
\tilde V_2(t):=&\tilde S_2(t)  (V^{(0)}_2(t_2))^{\top},
\end{align*}
where $\tilde S_2(t):=H(t)(S_2(t))^{-\top}$.

\begin{pro}
	The wave matrices satisfy
	\begin{align}\label{eq:central}
\big(V_1(t)\big)^{-1}\tilde V_2(t)=M.
\end{align}
\end{pro}
\begin{proof}
It is a consequence of  the Gaussian factorization.
\end{proof}
Given a semi-infinite matrix $A$ we have unique splitting $A=A_++A_-$ where $A_+$ is an upper  triangular block matrix while is $A_-$ a  strictly lower triangular block matrix.
\begin{pro}\label{pro:evolution S}
	The following equations hold
	\begin{align*}
	\frac{\partial S_1}{\partial t_{1,j,a}}(S_1)^{-1}&=-\Big(S_1E_{a,a}\Lambda^{j} (S_1)^{-1}\Big)_-, &
	\frac{\partial S_1}{\partial t_{2,j,a}}(S_1)^{-1}&=\Big(\tilde S_2E_{a,a}\big(\Lambda^{\top}\big)^j (\tilde S_2)^{-1}\Big)_-,\\
	\frac{\partial \tilde S_2}{\partial t_{1,j,a}}(\tilde S_2)^{-1}&=\Big(S_1E_{a,a}\Lambda^{j} (S_1)^{-1}\Big)_+, &
	\frac{\partial \tilde S_2}{\partial t_{2,j,a}}(\tilde S_2)^{-1}&=-\Big(\tilde S_2E_{a,a}\big(\Lambda^{\top}\big)^j (\tilde S_2)^{-1}\Big)_+.
	\end{align*}
\end{pro}
\begin{proof}
	Taking right derivatives  of \eqref{eq:central} yields
\begin{align*}
\frac{\partial V_1}{\partial t_{i,j,a}}(V_1)^{-1}&=\frac{\partial \tilde V_2}{\partial t_{i,j,a}}(\tilde V_2)^{-1}, & i&\in\{1,2\}, & j&\in\Z_+,
\end{align*}
where
\begin{align*}
\frac{\partial V_1}{\partial t_{1,j,a}}(V_1)^{-1}&=\frac{\partial S_1}{\partial t_{1,j,a}}(S_1)^{-1}+S_1E_{a,a}\Lambda^{j} (S_1)^{-1}, &
\frac{\partial V_1}{\partial t_{2,j,a}}(V_1)^{-1}&=\frac{\partial S_1}{\partial t_{2,j,a}}(S_1)^{-1},\\
\frac{\partial \tilde V_2}{\partial t_{1,j,a}}(\tilde V_2)^{-1}&=\frac{\partial \tilde S_2}{\partial t_{1,j,a}}(\tilde S_2)^{-1}, &
\frac{\partial \tilde V_2}{\partial t_{2,j,a}}(\tilde V_2)^{-1}&=\frac{\partial \tilde S_2}{\partial t_{2,j,a}}(\tilde S_2)^{-1}+\tilde S_2E_{a,a}\big(\Lambda^{\top}\big)^j (\tilde S_2)^{-1},
\end{align*}
and the result follows immediately.
\end{proof}
As a consequence, we derive
\begin{pro}
The multicomponent 2D Toda lattice equations
\begin{align*}
\frac{\partial}{\partial t_{2,1,b}}\Big(\frac{\partial H_k}{\partial t_{1,1,a}}(H_k)^{-1}\Big)+E_{a,a}H_{k+1}E_{b,b}(H_{k})^{-1}-H_kE_{b,b}(H_{k-1})^{-1}E_{a,a}=0.
\end{align*}
hold.
\end{pro}
\begin{proof}
	From Proposition \ref{pro:evolution S} we get
	\begin{align*}
	\frac{\partial H_k}{\partial t_{1,1,a}}(H_k)^{-1}&=\beta_kE_{a,a}-E_{a,a}\beta_{k+1},&
	\frac{\partial \beta_k}{\partial t_{2,1,b}}&=H_kE_{b,b}(H_{k-1})^{-1},
	\end{align*}
	where $\beta_k\in\R^{p\times p}$, $k=1,2,\dots$,  are the first subdiagonal coefficients in $S_1$.
\end{proof}

The multi-component Toda and KP hierarchies were introduced in \phantomsection\cite{ueno}.  In \phantomsection\cite{manas-1,manas0} its relevance in integrable aspects of differential geometry was emphasized, and in \phantomsection\cite{kac} a representation approach was developed, while in \phantomsection\cite{adler,amu} it was used in relation with multiple orthogonality.
A comprehensive approach to multi-component 2D Toda hierarchy  with applications in dispersionless integrability or generalized orthogonal polynomials  can be found in \phantomsection\cite{manas1,manas2,manas3}.

If we introduce the total flows given by the derivatives
\begin{align*}
\partial_{i,j}:=\sum_{a=1}^{p}\frac{\partial}{\partial t_{i,j,a}}
\end{align*}
we get the  non-Abelian 2D Toda lattice
\begin{align*}
\partial_{2,1}\big(\partial_{1,1}(H_k)\cdot (H_k)^{-1}\big)+H_{k+1}(H_{k})^{-1}-H_k(H_{k-1})^{-1}=0.
\end{align*}
The non-Abelian Toda lattice was introduced in the context of string theory by Polyakov, \phantomsection\cite{polyakov1,polyakov2}, and then studied under the inverse spectral transform by Mikhailov \phantomsection\cite{mikhailov} and Riemann surface theory by Krichever \phantomsection\cite{krichever}. The Darboux transformations were considered in \phantomsection\cite{salle} and later in \phantomsection\cite{nimmo}.

The non-Abelian 2D Toda lattice hierarchy is a reduction of the multicomponent hierarchy  by taking the diagonal time matrices $t_{i,j}=\diag(t_{i,j,1},\dots ,t_{i,j,p})$ proportional to the identity; i.e.,
\begin{align*}
t_{i,j}&\mapsto t_{i,j}I_p, & t_{i,j}&\in \R.
\end{align*}

These equations are just the first members  of an infinite set of nonlinear partial differential equations, an integrable hierarchy. Its elements are given by
\begin{defi}
	The partial, Lax and  Zakharov--Shabat matrices are given by
\begin{align*}
\Pi_{1,a}&:=S_1 E_{a,a} (S_1)^{-1}, &\Pi_{2,a}&:=\tilde S_2 E_{a,a} (\tilde S_2)^{-1}, \\
L_{1}&:=S_1\Lambda(S_1)^{-1}, &
 L_{2}&:=\tilde S_2\Lambda^{\top} (\tilde S_2)^{-1},\\
B_{1,j,a}&:=\big(\Pi_{1,a}(L_{1})^j\big)_+, & B_{2,j,a}&:=\big(\Pi_{2,a}(L_{2})^j\big)_-.
\end{align*}
\end{defi}
\begin{pro}[The integrable hierarchy]
The wave matrices obey the evolutionary linear systems
\begin{align*}
\frac{\partial V_1}{\partial t_{1,j,a}}&= B_{1,j,a}V_1,&\frac{\partial V_1}{\partial t_{2,j,a}}&= B_{2,j,a}V_1,\\
\frac{\partial \tilde V_2}{\partial t_{1,j,a}}&=B_{1,j,a}\tilde V_2, &\frac{\partial \tilde V_2}{\partial t_{2,j,a}}&= B_{2,j,a}\tilde V_2,
\end{align*}
the partial and Lax matrices are subject to the following  \emph{Lax equations}
\begin{align*}
\frac{\partial \Pi_{i',a'}}{\partial t_{i,j,a}}&=\Big[B_{i,j,a}, \Pi_{i',a'}\Big], &
\frac{\partial L_{i'}}{\partial t_{i,j,a}}&=\Big[B_{i,j,a}, L_{i'}\Big],
\end{align*}
and Zakharov--Sabat matrices fulfill the following  \emph{Zakharov--Shabat equations}
                         \begin{align*}
                       \frac{\partial B_{i',j',a'}}{\partial t_{i,j,a}}- \frac{\partial B_{i,j,a}}{\partial t_{i',j',a'}}+\big[B_{i,j,a},B_{i',j',a'}\big]=0.
                         \end{align*}
\end{pro}
      \begin{proof}
      	Follows from Proposition \ref{pro:evolution S}.
      \end{proof}
Given two semi-infinite block matrices      $A,B$ the notation $[A,B]=AB-BA$ stands for the usual commutator of matrices.

A crucial observation, regarding orthogonal polynomials, must be pointed out. When orthogonal polynomials are involved, and the matrices to  factorize are block Hankel, equivalently  $\Lambda M=M\Lambda^\top$, we get $   L_1=S_1\Lambda S_1^{-1}= \tilde S_2\Lambda^\top \tilde S_2^{-1}=L_2$.   As the reader may have noticed the Lax matrices $L_1$ and $L_2$ are, by construction,  lower and upper Hessenberg  block matrices, respectively.    However, when the Hankel property holds both Lax matrices are equal,
\begin{align*}
L_1=L_2,
\end{align*}
and, therefore, we are faced to a  tridiagonal block matrix; i.e., a Jacobi block matrix.  Moreover, this Hankel condition implies an invariance property under the flows introduced, as we have that $M(t)= V_1^{(0)}(t_1-t_2)M$, i.e., there are only one type of flows. This condition  also implies that for the total flows we have
\begin{align*}
(\partial_{1,j}+ \partial_{2,j})V_1 &=V_1 \Lambda^j, & (\partial_{1,j}+\partial_{2,j} )\tilde V_2 &=\tilde V_2 (\Lambda^\top)^j,\\
(\partial_{1,j}+ \partial_{2,j})L_1 &=0, & (\partial_{1,j}+\partial_{2,j} )L_2 &=0.
\end{align*}
Therefore, in the block Hankel case we are dealing with  the  multicomponent 1D Toda hierarchy.

\subsection{The Christoffel transformation for the non-Abelian 2D Toda hierarchy}

The idea is to follow  what we did  in \S \ref{S:connection} and consider an initial condition $\hat M$ at $t=0$, this is
	\begin{align*}
	\hat M=W(\Lambda)M
	\end{align*}
for a matrix polynomial $W(x)\in\R^{p\times p}[x]$.
Observe that using the scalar times $t_{i,j}\in\R$ of the non-Abelian flows determined by
 \begin{align*}
V_i^{(0)}&:=\exp\Big(\sum_{j=0}^\infty t_{i,j}\Lambda^j\Big), & i&=1,2,
\end{align*}
 the perturbed matrix is given by
\begin{align*}
\hat M(t)=& V_1^{(0)}(t_1)\hat M\big(V_2^{(0)}(t_2)\big)^{-\top}\\
=&W(\Lambda)M(t).
\end{align*}
Here we have used that $[W(\Lambda),V^{(0)}_1(t)]=0$, $\forall t_{1,j}\in\R$. Let us stress that we could  request only   $t_{1,j}$ to be scalars an let  $t_{2,j}$ to be diagonal matrices. Despite this is a more general situation,  we prefer to show how the method works in this simpler scenario.

Assuming that the block Gauss factorization hold, we proceed as in  \S \ref{S:connection}
and  introduce the resolvents
	\begin{align*}
	\omega^{[1]}(t)&:=\hat S_1(t)W(\Lambda)  \big(S_1(t)\big)^{-1}, & \omega^{[2]}(t)&:=(S_2(t)\big(\hat S_2(t)\big)^{-1})^\top.
	\end{align*}

From the $LU$ factorization we get
	\begin{align*}
	\big(\hat S_1(t)\big)^{-1}\hat H(t)\big(\hat S_2(t)\big)^{-\top}=W(\Lambda)  \big(S_1(t)\big)^{-1} H(t) \big(S_2(t)\big)^{-\top},
	\end{align*}
	so that
	\begin{align*}
	\hat H(t)( S_2(t)\big(\hat S_2(t)\big)^{-1})^\top=\hat S_1(t)W(\Lambda)  \big(S_1(t)\big)^{-1} H(t),
	\end{align*}
	and, consequently,
		\begin{align*}
		\hat H(t) \omega^{[2]}(t)=\omega^{[1]}(t) H(t)
		\end{align*}
	 holds.
Hence, as in the static case where the variable $t$ does not appear,  we have that this $t$-dependent resolvent matrix has a band  block upper triangular structure
	\begin{align*}
	\omega^{[1]}=\begin{bmatrix}
	\omega^{[1]}_{0,0} &\omega^{[1]}_{0,1} &\omega^{[1]}_{0,2}  &\dots &\omega^{[1]}_{0,N-1} &I_p&0&0&\dots \\
	0&\omega^{[1]}_{1,1} &\omega^{[1]}_{1,2} &\dots &\omega^{[1]}_{1,N-1} &\omega^{[1]}_{1,N} &I_p&0&\dots\\
	0&0&\omega^{[1]}_{2,2} &\dots &\omega^{[1]}_{2,N-1} &\omega^{[1]}_{2,N} &\omega^{[1]}_{2,N+1} &I_p&\ddots\\
	&\ddots&\ddots &\ddots& &&&\ddots&\ddots
	\end{bmatrix}
	\end{align*}
	with
	\begin{align*}
	\hat H_k(t)=\omega^{[1]}_{k,k}(t)H_k(t),
	\end{align*}
and the connection formulas described in Proposition \ref{pro:relation_omega} hold in this wider context.

Moreover, if $W(x)$ is a monic polynomial we can ensure that  the Christoffel formula is also fulfilled for the non-Abelian 2D Toda and Theorem \ref{theo:spectral} remains valid also in this scenario.
Formulas \eqref{eq:Christoffel1} and \eqref{eq:hat_H_k} hold directly and need no further explanation. However,
\eqref{eq:Christoffel2} needs the following brief discussion. The Christoffel--Darboux  kernel is defined exactly as we did  in \eqref{eq:CD kernel} but very probably there is no such a formula as the CD formula given in Proposition  \ref{pro:CD formula} is present in this scenario. However, as was shown in \phantomsection\cite{am}, there are cases, such as the multigraded reductions, where one has a generalized CD formula.

\end{document}